\documentclass{article}
\usepackage{amssymb,amsmath,amsthm,latexsym, enumerate}

\usepackage{latexsym}
\newtheorem{thm}{Theorem}[section]
\newtheorem{proposition}{Proposition}[section]
\newtheorem{definition}{Definition}[section]
\newtheorem{corollary}{Corollary}[section]
\newtheorem{lemma}{Lemma}[section]

\begin{document}

\title{Generalized Schur algebras}

\author{ Robert May \\
Department of Mathematics and Computer Science\\
Longwood University\\
201 High Street, Farmville, Va. 23909\\
rmay@longwood.edu}
\date{05/10/2014}
\maketitle

\section{Introduction}
    In \cite{MA} and \cite{May1}  ``generalized Schur algebras'', $B(n,r)$, were defined corresponding to a family of monoids $S_r $ which includes the full transformation semigroup $\tau _r = \mathcal{T}_r $, the partial transformation semigroup $\bar \tau_r = \mathcal{PT}_r $, and the rook monoid $\Re _r $ of size $r$.  In many cases a parameterization of the irreducible representations of the algebra $B\left( {n,r} \right)$ was obtained.  In \cite{May2} these algebras were redefined as the ``right generalized Schur algebras'' , $B_R $, corresponding to certain ``double coset algebras''.  Corresponding ``left generalized Schur algebras'' , $B_L $, were also defined in \cite{May2}.  In this paper we use the approach in \cite{May1} to study both the right and left algebras.  The algebras and their multiplication rules are described in sections 2 through 4.  In section 5, we obtain filtrations of these algebras.  These lead, in most cases, to parameterizations of the irreducible representations for both $B_R $ and $B_L $ for a field of characteristic 0 (in section 6) or positive characteristic $p$ (in sections 7 and 8).

\section{The semigroups $S_r$ and the double coset algebra $A(S_r,{\mathbf{G}})$}

Let $\bar \tau _r $ be the set of all maps $\alpha :\{ 0,1, \ldots ,r\}  \to \{ 0,1, \ldots ,r\} {\text{ such that }}\alpha (0) = 0$.  $\bar \tau _r $ is a semigroup under composition and is isomorphic to the partial transformation semigroup $\mathcal{PT}_r$.  Let $\tau _r $ be the full transformation semigroup and $\mathfrak{S}_r $ the symmetric group on $\{ 1,2, \ldots ,r\} $.  Any  $\alpha $ in $\tau _r$ or $\mathfrak{S}_r$ can be extended to $\bar \alpha  \in \bar \tau _r $ by defining $\bar \alpha (0) = 0$, so we can regard $\mathfrak{S}_r$ and $\tau _r $ as subsemigroups of $\bar \tau _r $. We will let $S_r$ represent any subsemigroup of $\bar \tau _r$ which contains $\mathfrak{S}_r $.  For example, $S_r $ could be the "rook semigroup", $\Re_r  = \left\{ {\alpha  \in \bar \tau _r :\forall i \in \{ 1,2, \ldots r\} ,\left| {\alpha ^{ - 1} (i)} \right| \leqslant 1} \right\}$.  Our main examples for $S_r $ will be $\mathfrak{S}_r ,\tau _r ,\Re_r ,{\text{ and }}\bar \tau _r $.  Note that $\mathfrak{S}_r  = \tau _r  \cap \Re_r {\text{ and }}\bar \tau _r  = \Re_r  \cdot \tau _r $.

Each $S_r $ can be identified with a certain semigroup of matrices.  Let $M_{r + 1} (\mathbb{Z})$ be the set of all $(r + 1) \times (r + 1)$ matrices with entries in $\mathbb{Z}$ (considered as a semigroup under matrix multiplication).  For convenience, label the rows and columns for each $m \in M_{r + 1} \left( \mathbb{Z} \right)$ from $0$ to $r$.  To each $\alpha  \in S_r $ assign a matrix $m(\alpha ) \in M_{r + 1} \left( \mathbb{Z} \right)
$ by setting $m(\alpha )_{i,j}  = \left\{ {\begin{array}{*{20}c}
   {1{\text{ if }}i = \alpha (j)}  \\
   {0{\text{ otherwise}}}  \\
 \end{array} } \right.$.  Note that column $j$ of $m(\alpha )$
 contains exactly one non-zero entry, namely a 1 in row $\alpha (j)$.  The 1 in column 0 is always in row 0.  It is not hard to check that $\alpha  \mapsto m(\alpha )$ gives an injective semigroup homomorphism $S_r  \to M_{r + 1} \left( \mathbb{Z} \right)$, so we will identify $S_r $ with its image in $M_{r + 1} (\mathbb{Z})$.  For $\alpha  \in S_r $we will usually write just $\alpha $ for the corresponding matrix $m(\alpha ){\text{ }}{\text{.}}$  If $\alpha \in R_r$, then $m(\alpha )$ has at most one 1 in rows $1,2,\dots,r$.   If $\alpha \in \tau _r$, then $m(\alpha )$ has only one 1 in row 0 (in column 0).  If $\alpha  \in \mathfrak{S}_r$, then $m(\alpha )$ is the usual permutation matrix bordered with an extra row 0 and column 0.

Let $\Lambda (r,n)$ be the set of all compositions of $r$ with $n$ parts.  So for $\lambda  \in \Lambda (r,n)$ we have $\lambda  = \left\{ {\lambda _i  \in \mathbb{Z}:i = 1,2, \ldots ,n} \right\},\,\,\lambda _i  \geqslant 0,\,\,\sum\limits_{i = 1}^n {\lambda _i }  = r.$  For each $\lambda  \in \Lambda (r,n)$ define ``$\lambda $-blocks'' of integers, $b_i ^\lambda  $, and a ``Young subgroup'', $\mathfrak{S}_\lambda   \subseteq \mathfrak{S}_r \subseteq S_r $ as follows.  First, let
\[b_1 ^\lambda   = \left\{ {k \in \mathbb{Z}:0 < k \leqslant \lambda _1 } \right\} \text{and} \]
\[b_i ^\lambda   = \left\{ {k \in \mathbb{Z}:\lambda _1  + \lambda _2  +  \cdots  + \lambda _{i - 1}  < k \leqslant \lambda _1  + \lambda _2  +  \cdots  + \lambda _i } \right\} \text{for } 1 < i \leqslant n.\]  Thus $
b_i ^\lambda  $ consists of $\lambda _i $ consecutive integers and $
b_i ^\lambda   = \emptyset  \Leftrightarrow \lambda _i  = 0$.  Now let $\mathfrak{S}\left( {b_i ^\lambda  } \right) \subseteq \mathfrak{S}_r $ be the group of all permutations of $b_i ^\lambda  $, so $\mathfrak{S}\left( {b_i ^\lambda  } \right) \cong \mathfrak{S}_{\lambda _i } $.  (Put $
\mathfrak{S}\left( {b_i ^\lambda  } \right) = \left\{ {{\text{identity element in }}\mathfrak{S}_r } \right\}{\text{ when }}b_i ^\lambda   = \emptyset $.)  Finally, define the Young subgroup by $\mathfrak{S}_\lambda   = \prod\limits_{i = 1}^r {\mathfrak{S}\left( {b_i ^\lambda  } \right)}  \cong \prod\limits_{i = 1}^r {\mathfrak{S}_{\lambda _i} }  $.  Notice that each $\mathfrak{S}_\lambda  $ is generated by $\left\{ {s_i :s_i  \in \mathfrak{S}_\lambda  } \right\}$ where $s_i  \in \mathfrak{S}_r$ is the elementary transposition which interchanges $i{\text{ and }}i + 1$.  We often write $G_\lambda  $ for the Young subgroup $\mathfrak{S}_\lambda  $.

For $\lambda ,\mu  \in \Lambda \left( {n,r} \right)$, write $_\lambda  M$ for the set of left  $\mathfrak{S}_\lambda  $-cosets of the form $\mathfrak{S}_\lambda  \alpha \,,\,\alpha  \in S_r $, $M_\mu  $
for the set of right $\mathfrak{S}_\mu  $-cosets  $\alpha \mathfrak{S}_\mu  \,,\,\alpha  \in S_r $, and $_\lambda  M_\mu  $ for the set of all double cosets $\mathfrak{S}_\lambda  \alpha \mathfrak{S}_\mu  \,,\,\alpha  \in S_r $.  Write $D\left( {\lambda ,\alpha ,\mu } \right)$ for the double coset $\mathfrak{S}_\lambda  \alpha \mathfrak{S}_\mu  $.  We may write just $D_\alpha  $ for the double coset $\mathfrak{S}_\lambda  \alpha \mathfrak{S}_\mu  $ when $\lambda ,\mu $ are clear.

Let ${\mathbf{G}} = \left\{ {G_\lambda   = \mathfrak{S}_\lambda  :\lambda  \in \Lambda \left( {r,n} \right)} \right\}$, a family of subgroups of the monoid $S_r $.  The double coset algebra $A = A\left( {S_r ,{\mathbf{G}}} \right)$ was defined in \cite{May2}.  $A$ is a complex vector space with basis $\left\{ {X\left( {\lambda ,D,\mu } \right)\,:\,\lambda ,\mu  \in \Lambda \left( {n,r} \right)\,,\,D \in \,_\lambda  M_\mu  } \right\}$.
The product in $A$ is given by 
    \[X(\lambda ,D_1 ,\nu ) \cdot X\left( {\omega ,D_2 ,\mu } \right) =
    \begin{cases}
    0, &\text{if $\nu  \ne \omega$;}\\
    {\sum\limits_{D \in \,_\lambda  M_\mu  } {a\left( {\lambda ,\mu ,D_1 ,D_2 ,D} \right)X\left( {\lambda ,D,\mu } \right)} }, &\text{if $\nu  = \omega $.}
    \end{cases}
    \]
If $D = D\left( {\lambda ,m,\mu } \right)$, then $a\left( {\lambda ,\mu ,D_1 ,D_2 ,D} \right)$ is the nonnegative integer
\[a\left( {\lambda ,\mu ,D_1 ,D_2 ,D} \right) = \# \left\{ {\left( {m_1 ,m_2 } \right) \in D_1  \times D_2 :m_1 m_2  = m} \right\}.\]

As shown in \cite{May2}, there is an alternative formula for the structure constants $a\left( {\lambda ,\mu ,D_1 ,D_2 ,D} \right)$.   For any double coset $D \in \,_\lambda  M_\mu  $ , the product group $G_\lambda   \times G_\mu  $ acts transitively on $D$ by $(\sigma  \times \pi )(m) = \sigma m\pi ^{ - 1} $ for $m \in D,\sigma  \in G_\lambda  ,\pi  \in G_\mu  $.  Then for any $m \in D$, let $n\left( D \right)$ be the order of the subgroup $G_{\lambda ,\mu ,m} $ of $G_\lambda   \times G_\mu  $ which leaves $m$ fixed:  $n(D) = \# \left\{ {\left( {\sigma  \times \pi } \right):\sigma  \in G_\lambda  ,\pi  \in G_\mu  ,\sigma m\pi ^{ - 1}  = m} \right\}$.  It is easy to check that $G_{\lambda ,\mu ,m} \,,\,G_{\lambda ,\mu ,n} $ are conjugate subgroups whenever $m,n \in D$, so the order $n\left( D \right)$ is independent of the choice of $m \in D$.  For $D_m  \in \,_\lambda  M_\nu  ,D_n  \in \,_\nu  M_\mu  ,D \in \,_\lambda  M_\mu  $  , put $N\left( {D_m ,D_n ,D} \right) = \# \left\{ {\rho  \in G_\nu  :m\rho n \in D} \right\}$ (which is independent of the choices of $m,n$).  Finally, let $o\left( {G_\nu  } \right)$ be the order of the group $G_\nu  $.  Then the structure constants are given by
     \[a\left( {\lambda ,\mu ,D_m ,D_n ,D} \right) = \frac{{o\left( {G_\nu  } \right)N\left( {D_m ,D_n ,D} \right)n(D)}}
{{n\left( {D_m } \right)n\left( {D_n } \right)}}.\]

When $S_r  = \mathfrak{S}_r $, the algebra $A$ is isomorphic to the classical complex Schur algebra $S\left( {r,n} \right)$, so for general $S_r $ we call $A$ the generalized Schur algebra associated with $S_r $.

\section{The left and right $\mathbb{Z}$-forms for $A$}

In \cite{May2}, the left and right $\mathbb{Z}$-forms for the algebra $A = A\left( {S_r ,{\mathbf{G}}} \right)$ were defined.  These algebras, denoted here by $A_L ,\,A_R $ , are $\mathbb{Z}$-algebras with unit such that there are isomorphisms of complex algebras $A_L \mathop  \otimes_\mathbb{Z} \mathbb{C} \cong A \cong A_R \mathop  \otimes_\mathbb{Z} \mathbb{C}$ .  Each of $A_L ,\,A_R $ is a free $\mathbb{Z}$-module with a basis $\left\{ {f\left( {\lambda ,D,\mu } \right)\,:\,\lambda ,\mu  \in \Lambda \left( {n,r} \right)\,,\,D \in \,_\lambda  M_\mu  } \right\}$ corresponding to the collection of all double cosets.  We now describe the (integer) structure constants for these algebras as given in \cite{May2}.
     For $D \in \,_\lambda  M_\mu  $, note that $G_\mu  $ acts transitively on the set of left cosets $C \in \,_\lambda  M,C \subseteq D$.  Let $G_L \left( {\lambda ,C,\mu } \right) = \left\{ {\pi  \in G_\mu  :C\pi  = C} \right\}$ be the subgroup of $G_\mu  $ which fixes a given $C$.  If $C = G_\lambda  \alpha $, then $G_L (\lambda ,C,\mu ) = \left\{ {\pi  \in G_\mu  :\exists \sigma  \in G_\lambda  {\text{ s}}{\text{.t}}{\text{. }}\alpha \pi  = \sigma \alpha } \right\}$ and we write $G_L \left( {\lambda ,\alpha ,\mu } \right) = G_L \left( {\lambda ,C,\mu } \right) \subseteq G_\mu  $.  For any two left cosets $
C_1 ,C_2  \in \,_\lambda  M\,,\,\,C_1 ,C_2  \subseteq D$ the subgroups $G_L \left( {C_1 } \right),G_L (C_2 )$ are conjugate, so we can define $n_L \left( D \right)$ to be the order of $G_L \left( {\lambda ,C,\mu } \right)$ for any such $C$.  If $D = D_\alpha   = G_\lambda  \alpha G_\mu  $, we write  $n_L \left( {\lambda ,\alpha ,\mu } \right) = n_L (D)$.  
     Similarly, for any right coset $C = \alpha G_\mu   \subseteq D$ , let $G_R \left( {\lambda ,C,\mu } \right) = \left\{ {\sigma  \in G_\lambda  :\sigma C = C} \right\} = \left\{ {\sigma  \in G_\lambda  :\exists \pi  \in G_\mu  {\text{ s}}{\text{.t}}{\text{. }}\sigma \alpha  = \alpha \pi } \right\}$ be the subgroup of $G_\lambda  $ which fixes $C$ and let $n_R \left( D \right)$ be the order of $G_R \left( {\lambda ,C,\mu } \right)$ (independent of the choice of $C$).  We also write $G_R \left( {\lambda ,\alpha ,\mu } \right) = G_R \left( {\lambda ,C,\mu } \right) \subseteq G_\lambda  $.  If $D = D_\alpha   = G_\lambda  \alpha G_\mu  $, we write $n_R \left( {\lambda ,\alpha ,\mu } \right) = n_R (D)$.  
     Let $ * _L , * _R $ be the products in $A_L ,A_R $ respectively.  Write $f\left( {\lambda ,\alpha ,\mu } \right)$ for the basis element $f\left( {\lambda ,D_\alpha  ,\mu } \right)$.  Then from \cite{May2} we have the following
\\

\textbf{Multiplication Rules:}
	\[f\left( {\lambda ,\alpha ,\nu } \right) * _R f\left( {\omega ,\beta ,\mu } \right) = f\left( {\lambda ,\alpha ,\nu } \right) * _L f\left( {\omega ,\beta ,\mu } \right) = 0{\text{ if }}\nu  \ne \omega \]

\[f\left( {\lambda ,\alpha ,\nu } \right) * _R f\left( {\nu ,\beta ,\mu } \right) = \sum\limits_{D \in \,_\lambda  M_\mu  } {\frac{{n_R (D)N\left( {D_\alpha  ,D_\beta  ,D} \right)}}
{{n_R \left( {\lambda ,\alpha ,\nu } \right)n_R \left( {\nu ,\beta ,\mu } \right)}}f\left( {\lambda ,D,\mu } \right)} \]

\[f\left( {\lambda ,\alpha ,\nu } \right) * _L f\left( {\nu ,\beta ,\mu } \right) = \sum\limits_{D \in \,_\lambda  M_\mu  } {\frac{{n_L (D)N\left( {D_\alpha  ,D_\beta  ,D} \right)}}
{{n_L \left( {\lambda ,\alpha ,\nu } \right)n_L \left( {\nu ,\beta ,\mu } \right)}}f\left( {\lambda ,D,\mu } \right)} \]

As shown in \cite{May2}, the structure constants appearing in the multiplication rules are always nonnegative integers.

An alternative form of the multiplication rules is sometimes useful.
\begin{lemma} \label{l3.1}
\begin{enumerate} [\upshape(a)]
\item if $\rho _1 ,\rho _2 $ are in the same right $G_R \left( {\nu ,\beta ,\mu } \right)$-coset of $G_\nu  $, then for any $\alpha  \in S_r $, $\alpha \rho _1 \beta$ and $\alpha \rho _2 \beta $ are in the same $G_\lambda - G_\mu $ double coset.
\item if $\rho _1 ,\rho _2 $ are in the same left $G_L \left( {\lambda ,\alpha ,\nu } \right)$-coset of $G_\nu  $, then for any $\beta  \in S_r $, $\alpha \rho _1 \beta {\text{ and }}\alpha \rho _2 \beta $ are in the same $G_\lambda   - G_\mu  $ double coset. 
\end{enumerate}
\end{lemma}

\begin{proof}
  a)  Suppose $\rho _2  = \rho _1 \kappa $ for some $\kappa  \in G_R \left( {\nu ,\beta ,\mu } \right)$.  Then there exists $\pi  \in G_\mu  $ such that $\kappa \beta  = \beta \pi $.  Then $\alpha \rho _2 \beta  = \alpha \rho _1 \kappa \beta  = \alpha \rho _1 \beta \pi  \in G_\lambda  \left( {\alpha \rho _1 \beta } \right)G_\mu  $.
  
b)  Suppose $\rho _2  = \kappa \rho _1 $ for some $\kappa  \in G_L \left( {\lambda ,\alpha ,\nu } \right)$.  Then there exists $\sigma  \in G_\lambda  $ such that $\sigma \alpha  = \alpha \kappa $.  Then $\alpha \rho _2 \beta  = \alpha \kappa \rho _1 \beta  = \sigma \alpha \rho _1 \beta  \in G_\lambda  \left( {\alpha \rho _1 \beta } \right)G_\mu  $. 
\end{proof}

It follows that for any right $G_R \left( {\nu ,\beta ,\mu } \right)$
 coset, any $D \in \,_\lambda  M_\mu  $, and any $\alpha  \in S_r $, either $\alpha \rho \beta  \in D$ for every $\rho $ in the coset or $\alpha \rho \beta  \notin D$ for any $\rho $ in the coset (and similarly for left $G_L \left( {\lambda ,\alpha ,\nu } \right)$ cosets) .  

\begin{proposition} \label{p3.1}
Let $D \in \,_\lambda  M_\mu  $.  Then
\[N\left( {D\left( {\lambda ,\alpha ,\nu } \right),D\left( {\nu ,\beta ,\mu } \right),D} \right) = a_{R,D} n_R \left( {\nu ,\beta ,\mu } \right) = a_{L,D} n_L \left( {\lambda ,\alpha ,\nu } \right)\] where $a_{R,D}  = $ the number of distinct right cosets $\rho G_R \left( {\nu ,\beta ,\mu } \right)$ in $G_\nu  $ such that $\alpha \rho \beta  \in D$, and $a_{L,D}  = $ the number of distinct left cosets $G_L \left( {\lambda ,\alpha ,\nu } \right)\rho $ in $G_\nu  $ such that $\alpha \rho \beta  \in D$.
\end{proposition}

\begin{proof}
  All right cosets $\rho G_R \left( {\nu ,\beta ,\mu } \right)$
 have $o\left( {G_R (\nu ,\beta ,\mu )} \right) = n_R \left( {\nu ,\beta ,\mu } \right)$ elements.  So \[N\left( {D\left( {\lambda ,\alpha ,\nu } \right),D\left( {\nu ,\beta ,\mu } \right),D} \right) = \# \left\{ {\rho  \in G_\nu  :\alpha \rho \beta  \in D} \right\} = a_{R,D} n_R \left( {\nu ,\beta ,\mu } \right).\]  Similarly, all left cosets $G_L \left( {\lambda ,\alpha ,\nu } \right)\rho $ have $n_L \left( {\lambda ,\alpha ,\nu } \right)$ elements, so \[N\left( {D\left( {\lambda ,\alpha ,\nu } \right),D\left( {\nu ,\beta ,\mu } \right),D} \right) = \# \left\{ {\rho  \in G_\nu  :\alpha \rho \beta  \in D} \right\} = a_{L,D} n_L \left( {\lambda ,\alpha ,\nu } \right) .\]
\end{proof}

Substituting into the multiplication rules gives the following
\\

\textbf{Alternative forms of multiplication rules:}
\[f\left( {\lambda ,\alpha ,\nu } \right) * _R f\left( {\nu ,\beta ,\mu } \right) = \sum\limits_{D \in \,_\lambda  M_\mu  } {\frac{{n_R (D)\,\,a_{R,D} }}{{n_R \left( {\lambda ,\alpha ,\nu } \right)}}f\left( {\lambda ,D,\mu } \right)} \]

\[f\left( {\lambda ,\alpha ,\nu } \right) * _L f\left( {\nu ,\beta ,\mu } \right) = \sum\limits_{D \in \,_\lambda  M_\mu  } {\frac{{n_L (D)\,\,a_{L,D} }}{{n_L \left( {\nu ,\beta ,\mu } \right)}}f\left( {\lambda ,D,\mu } \right)} .\]

Some special cases of the multiplication rules will also be useful.
\\[12pt]
\textbf{Case 1:}  $\nu _i  =
\begin{cases}
1,  &\text{for $1 \leqslant i \leqslant r$,}\\
0,  &\text{for $i > r$.}
\end{cases}$
\\[12pt]
Then $G_\nu   = \left\{ {id} \right\}$ and $n_R \left( {\nu ,\beta ,\mu } \right) = 1 = n_L \left( {\lambda ,\alpha ,\nu } \right)$ for any $\lambda ,\mu  \in \Lambda \left( {r,n} \right)$ and $\alpha ,\beta  \in S_r $.  Also
\[N\left( {D\left( {\lambda ,\alpha ,\nu } \right),D\left( {\nu ,\beta ,\mu } \right),D} \right) =
    \begin{cases}
    1,  &\text{if $D = G_\lambda  \alpha \beta G_\mu$;}\\
    0,  &\text{otherwise.}
    \end{cases}    \]
Then substituting in the multiplication rule gives
	\[f\left( {\lambda ,\alpha ,\nu } \right) * _R f\left( {\nu ,\beta ,\mu } \right) = \frac{{n_R \left( {D\left( {\lambda ,\alpha \beta ,\mu } \right)} \right)}}
{{n_R \left( {D\left( {\lambda ,\alpha ,\nu } \right)} \right)}}f\left( {\lambda ,\alpha \beta ,\mu } \right)\]
For $\alpha  = 1$ (the identity in $S_r $), $n_R \left( {\lambda ,1,\nu } \right) = 1$ giving	
\[f\left( {\lambda ,1,\nu } \right) * _R f\left( {\nu ,\beta ,\mu } \right) = n_R \left( {D\left( {\lambda ,\beta ,\mu } \right)} \right)f\left( {\lambda ,\beta ,\mu } \right)\]
If in addition $\lambda  = \nu $, then writing $1_\nu   = f\left( {\nu ,1,\nu } \right)$ gives 
\[1_\nu   * _R f\left( {\nu ,\beta ,\mu } \right) = f\left( {\nu ,\beta ,\mu } \right).\]
There are similar rules for $ * _L $:
\[f\left( {\lambda ,\alpha ,\nu } \right) * _L f\left( {\nu ,\beta ,\mu } \right) = \frac{{n_L \left( {D\left( {\lambda ,\alpha \beta ,\mu } \right)} \right)}}{{n_L \left( {D\left( {\lambda ,\alpha ,\nu } \right)} \right)}}f\left( {\lambda ,\alpha \beta ,\mu } \right)\]
\[f\left( {\lambda ,\alpha ,\nu } \right) * _L f\left( {\nu ,1,\mu } \right) = n_L \left( {D\left( {\lambda ,\alpha ,\mu } \right)} \right)f\left( {\lambda ,\alpha ,\mu } \right)\]
\[f\left( {\lambda ,\alpha ,\nu } \right) * _L 1_\nu   = f\left( {\lambda ,\alpha ,\nu } \right).\]
\\[12pt]
\textbf{Case 2:}  $\alpha  = 1$ (the identity in $S_r $) and $G_\nu   \subseteq G_\lambda  $.
\\[12pt]
Then for any $\rho \in G_\nu$, we have $\alpha \rho \beta  = \rho \beta  \in G_\lambda  \beta G_\mu  $, so
  \[N\left( {D\left( {\lambda ,1,\nu } \right),D\left( {\nu ,\beta ,\mu } \right),D} \right) =
  \begin{cases}
  o\left( {G_\nu  } \right),   &\text{if $D = G_\lambda  \beta G_\mu$;}\\
   0,   &\text{otherwise.}
  \end{cases}\]
Also $G_R \left( {\lambda ,1,\nu } \right) = G_\nu   = G_L \left( {\lambda ,1,\nu } \right)$, so 
\[n_R \left( {D\left( {\lambda ,1,\nu } \right)} \right) = o\left( {G_\nu  } \right) = n_L \left( {D\left( {\lambda ,1,\nu } \right)} \right).\]
Then substitution in the multiplication rule gives:
\[f\left( {\lambda ,1,\nu } \right) * _R f\left( {\nu ,\beta ,\mu } \right) = \frac{{n_R \left( {D\left( {\lambda ,\beta ,\mu } \right)} \right)}}
{{n_R \left( {D\left( {\nu ,\beta ,\mu } \right)} \right)}}f\left( {\lambda ,\beta ,\mu } \right).\]
 In particular, if $\lambda  = \nu $ and $1_\nu   = f\left( {\nu ,1,\nu } \right)$ then
 \[1_\nu   * _R f\left( {\nu ,\beta ,\mu } \right) = f\left( {\nu ,\beta ,\mu } \right).\]
Similarly, 
\[f\left( {\lambda ,1,\nu } \right) * _L f\left( {\nu ,\beta ,\mu } \right) = \frac{{n_L \left( {D\left( {\lambda ,\beta ,\mu } \right)} \right)}}
{{n_L \left( {D\left( {\nu ,\beta ,\mu } \right)} \right)}}f\left( {\lambda ,\beta ,\mu } \right)\]
\[1_\nu   * _L f\left( {\nu ,\beta ,\mu } \right) = f\left( {\nu ,\beta ,\mu } \right).\]
\\[12pt]
\textbf{Case 3:}  $\beta  = 1$ and $G_\nu   \subseteq G_\mu  $.
\\[12pt]
  Then for any $\rho \in G_\nu$, we have $\alpha \rho \beta  = \alpha \rho  \in G_\lambda  \alpha G_\mu  $, so
  \[N\left( {D\left( {\lambda ,\alpha ,\nu } \right),D\left( {\nu ,1,\mu } \right),D} \right) =
  \begin{cases}
  o\left( {G_\nu  } \right),   &\text{if $D = G_\lambda  \alpha G_\mu$;}\\
  0    &\text{otherwise.}
  \end{cases}
  \]
Also $G_R \left( {\nu ,1,\mu } \right) = G_\nu   = G_L \left( {\nu ,1,\mu } \right)$, so 
\[n_R \left( {D\left( {\nu ,1,\mu } \right)} \right) = o\left( {G_\nu  } \right) = n_L \left( {D\left( {\nu ,1,\mu } \right)} \right).\]
Then substitution in the multiplication rule gives:
\[f\left( {\lambda ,\alpha ,\nu } \right) * _R f\left( {\nu ,1,\mu } \right) = \frac{{n_R \left( {D\left( {\lambda ,\alpha ,\mu } \right)} \right)}}{{n_R \left( {D\left( {\lambda ,\alpha ,\nu } \right)} \right)}}f\left( {\lambda ,\alpha ,\mu } \right)\]
In particular, if $\mu  = \nu $ and $1_\nu   = f\left( {\nu ,1,\nu } \right)$ then
\[f\left( {\lambda ,\alpha ,\nu } \right) * _R 1_\nu   = f\left( {\lambda ,\alpha ,\nu } \right).\]
Similarly, 
\[f\left( {\lambda ,\alpha ,\nu } \right) * _L f\left( {\nu ,1,\mu } \right) = \frac{{n_L \left( {D\left( {\lambda ,\alpha ,\mu } \right)} \right)}}{{n_L \left( {D\left( {\lambda ,\alpha ,\nu } \right)} \right)}}f\left( {\lambda ,\alpha ,\mu } \right)\]
\[f\left( {\lambda ,\alpha ,\nu } \right) * _L 1_\nu   = f\left( {\lambda ,\alpha ,\nu } \right).\]
\\[12pt]
\textbf{Case 4:} $\alpha  = \beta  = 1\,,\,\lambda  = \mu \,,\,G_\nu   \subseteq G_\lambda   = G_\mu $. 
\\[12pt]
As a special case of either case 2 or case 3, 
\[f\left( {\lambda ,1,\nu } \right) * _R f\left( {\nu ,1,\lambda } \right) = \frac{{o\left( {G_\lambda  } \right)}}
{{o\left( {G_\nu  } \right)}}f\left( {\lambda ,1,\lambda } \right) = f\left( {\lambda ,1,\nu } \right) * _L f\left( {\nu ,1,\lambda } \right).\]
In particular, when $\nu  = \lambda  = \mu $,
\[1_\nu   * _R 1_\nu   = 1_\nu   = 1_\nu   * _L 1_\nu .\]
\\[12pt]
So in both $A_R $ and $A_L $, each $1_\nu   = f\left( {\nu ,1,\nu } \right)$ is an idempotent.   In fact for either algebra, cases 2 and 3 can be used to check that  $\left\{ {1_\nu  :\nu  \in \Lambda \left( {r,n} \right)} \right\}$ is a family of orthogonal idempotents with $\mathop  \oplus \limits_{\nu  \in \Lambda \left( {r,n} \right)} 1_\nu   = 1$ (the identity).

\section{The left and right generalized Schur algebras}

     Let $R$ be a commutative domain with unit $1$ and let $\psi :\mathbb{Z} \to R$ be the natural ring homomophism such that $\psi \left( 1 \right) = 1$.  Regard $R$ as a right $\mathbb{Z}$-module via $\psi $ and form the tensor products $LGS_R  = R\mathop  \otimes_\mathbb{Z} A_L $ and  $RGS_R  = R\mathop  \otimes_\mathbb{Z} A_R$ which we call the left and right generalized Schur algebras over $R$ (for the monoid $S_r $).  (As in \cite{May2}, for $S_r  = \mathfrak{S}_r $, $LGS_R $ is isomorphic to the classical Schur algebra over $R$, $S_R \left( {r,n} \right)$, while $RGS_R $ is isomorphic to the opposite algebra $S_R \left( {r,n} \right)^{op} $.)  The $R$ algebras $LGS_R \,,\,RGS_R $ are both free as $R$ modules with bases $\left\{ {1 \otimes f\left( {\lambda ,D,\mu } \right):\lambda ,\mu  \in \Lambda \left( {r,n} \right),D \in \,_\lambda  M_\mu  } \right\}$.  We will often denote a basis element $1 \otimes f\left( {\lambda ,D,\mu } \right)$ by just $f\left( {\lambda ,D,\mu } \right)$.  Then all the multiplication rules given in section 3 apply if we identify each coefficient $a \in \mathbb{Z}$ occurring with its image $\psi \left( a \right) \in R$.
      As shown in \cite{May2}, the isomorphism $A \cong \mathbb{C} \otimes A_L  \equiv LGS_\mathbb{C} $ is given by matching basis elements $\frac{{X\left( {\lambda ,D,\mu } \right)}}
{{n_L \left( D \right)}} \leftrightarrow 1 \otimes f\left( {\lambda ,D,\mu } \right)$, where $n_L \left( D \right)$ is the number of elements in any left $G_\lambda  $ coset $C \subseteq D \in \,_\lambda  M_\mu  $.  Similarly, the isomorphism $A \cong \mathbb{C} \otimes A_R  \equiv RGS_\mathbb{C} $ is given by matching basis elements $\frac{{X\left( {\lambda ,D,\mu } \right)}}{{n_R \left( D \right)}} \leftrightarrow 1 \otimes f\left( {\lambda ,D,\mu } \right)$, where $n_R \left( D \right)$
 is the number of elements in any right $G_\mu  $ coset $C \subseteq D \in \,_\lambda  M_\mu  $.

\section{Filtrations of generalized Schur algebras and irreducible representations}

Let $\Lambda  = \Lambda (r,n)$ be the set of all compositions of $r$ into $n$ parts; $\Lambda ^ +   = \Lambda ^ +  (r) \subseteq \Lambda $, the set of all partitions of $r$.  For $\lambda  \in \Lambda $ and $0 \leqslant k \leqslant r$, put $L(\lambda ,k) = \# \left\{ {i:\lambda _i  = k} \right\}$ (= number of rows of length $k$ in $\lambda $ ).   Define an equivalence relation on $\Lambda $ by $\lambda _1  \sim \lambda _2  \Leftrightarrow \forall i,L\left( {\lambda _1 ,i} \right) = L\left( {\lambda _2 ,i} \right)$.  Notice that for each $\lambda  \in \Lambda $ there is exactly one partition $\lambda ^ +   \in \Lambda ^ +  $ with $\lambda  \sim \lambda ^ +  $.   Let $ < $ be the ``lexicographic'' order on $\Lambda ^ +  $:   For $\lambda ,\mu  \in \Lambda ^ +  $, $\lambda  < \mu  \Leftrightarrow \exists k{\text{ s}}{\text{.t}}{\text{. }}\lambda _i  = \mu _i {\text{ for }}i < k$ while $\lambda _k  < \mu _k $.  Then $ < $  gives a total ordering of the partitions or of the equivalence classes of compositions.  It extends to a partial order on $\Lambda $:  for $\lambda ,\mu  \in \Lambda \,,\,\,\lambda  \leqslant \mu  \Leftrightarrow \lambda  \sim \mu {\text{ or }}\lambda ^ +   < \mu ^ +  $ where $\lambda ^ +  ,\mu ^ +   \in \Lambda ^ +  $ are the partitions corresponding to $\lambda ,\mu $.  (This partial order differs slightly from that used in \cite{May2} .)  The smallest partition, $\bar \nu  = 1^{(r)} $, has $\,\,\bar \nu _i  = 1$ for all i.  The largest partition, $\lambda  = (r)$, has $\lambda _1  = r,\lambda _i  = 0,i > 1$.

Write $B$ for either of the $R$-algebras $LGS_R $ or $RGS_R $.  As mentioned above, $\left\{ {1_\lambda   = f\left( {\lambda ,1,\lambda } \right):\lambda  \in \Lambda } \right\}$ is a set of orthogonal idempotents in $B$ with $\mathop  \oplus \limits_{\lambda  \in \Lambda } 1_\lambda   = 1$ (the identity in $B$).  For each partition $\lambda  \in \Lambda ^ +  $ define $e_\lambda   = \mathop  \oplus \limits_{\mu  \in \Lambda ,\mu  < \lambda } 1_\mu  $ and $\bar e_\lambda   = e_\lambda   \oplus \left( {\mathop  \oplus \limits_{\mu  \in \Lambda ,\mu  \sim \lambda } 1_\mu  } \right)$.  (For the smallest partition $\bar \nu  = 1^{(r)} $ define $e_{\bar \nu }  = 0$.)   Note the following results:  each $e_\lambda  {\text{ and }}\bar e_\lambda  $ is an idempotent; for the largest partition, $(r)$, we have $\bar e_{(r)}  = 1$ (the identity in $B$);  $\lambda _1 ,\lambda _2  \in \Lambda ^ +   \Rightarrow \bar e_{\lambda _1 } \bar e_{\lambda _2 }  = \bar e_{\min (\lambda _1 ,\lambda _2 )} {\text{ and }}e_{\lambda _1 } e_{\lambda _2 }  = e_{\min (\lambda _1 ,\lambda _2 )} $ ;  $\bar e_\lambda  e_\lambda   = e_\lambda  \bar e_\lambda   = e_\lambda  $.

Obtain a filtration of the algebra $B$ by defining $B^\lambda   = \bar e_\lambda  B\bar e_\lambda  ,\lambda  \in \Lambda ^ +  $.  Then $B^{(r)}  = B$ and $\lambda _1  < \lambda _2  \Rightarrow B^{\lambda _1 }  \subseteq B^{\lambda _2 } $.  Each $B^\lambda  $ is an $R$-algebra with unit $\bar e_\lambda  $.  Next define $K^\lambda   = \left( {Be_\lambda  B} \right) \cap B^\lambda   = B^\lambda  e_\lambda  B^\lambda  $.  Then $K^\lambda  $ is the two-sided ideal in $B^\lambda  $ generated by $e_\lambda  $.   
Finally, define $Q^\lambda   = B^\lambda  /K^\lambda  $.  Then $Q^\lambda  $ is an $R$-algebra with unit $\bar e_\lambda  \,\bmod \,K^\lambda  $ .

By standard idempotent theory, if $I$ is an irreducible (left) $B$-module, then either $\bar e_\lambda  I = 0$ or $\bar e_\lambda  I$ is an irreducible $B^\lambda  \left( { = \bar e_\lambda  B\bar e_\lambda  } \right)$-module.  Define an irreducible $B$-module $I$ to be at level $\lambda $ if $\bar e_\lambda  I \ne 0{\text{ and }}e_\lambda  I = 0$.  Then isomorphism classes of irreducible $B$-modules at levels $ \leqslant \lambda $ correspond to those with $\bar e_\lambda  I \ne 0$, which in turn correspond to isomorphism classes of irreducible $B^\lambda  $-modules.

An irreducible $Q^\lambda  $-module corresponds to an irreducible $B^\lambda  $-module on which $K^\lambda  $ acts trivially, and thus to an irreducible $B$-module $I$ such that $\bar e_\lambda  I \ne 0{\text{ but }}K^\lambda  \bar e_\lambda  I = 0$.  But then $K^\lambda  \bar e_\lambda  I = 0 \Rightarrow e_\lambda  \bar e_\lambda  I = 0 \Rightarrow e_\lambda  I = 0 \Rightarrow I$  is at level $\lambda $.  Since any irreducible $B$-module lies at exactly one level, it corresponds to an irreducible $Q^\lambda  $-module for exactly one $\lambda  \in \Lambda ^ +  $.  
Thus a classification of the irreducible $Q^\lambda  $-modules for all $\lambda  \in \Lambda ^ +  $ will give a classification of all irreducible $B$-modules.

We will make one further reduction.  Let $\bar B^\lambda   = \left( {1_\lambda   + e_\lambda  } \right)B\left( {1_\lambda   + e_\lambda  } \right) \subseteq B^\lambda  $.  This is an $R$-algebra with unit $1_\lambda   + e_\lambda  $.  Let $\bar K^\lambda   = K^\lambda   \cap \bar B^\lambda   = \bar B^\lambda  e_\lambda  \bar B^\lambda  $, the two-sided ideal in $\bar B^\lambda  $ generated by $e_\lambda  $, and define $C^\lambda   = \bar B^\lambda  /\bar K^\lambda  $.  $C^\lambda  $ is an $R$-algebra with unit $1_\lambda  \bmod \bar K^\lambda  $ and can be identified with a subalgebra of $Q^\lambda  $:  $C^\lambda   = \left( {1_\lambda  \bmod K^\lambda  } \right)Q^\lambda  \left( {1_\lambda  \bmod K^\lambda  } \right)$.  We will see that irreducible $Q^\lambda  $-modules correspond to irreducible $C^\lambda  $-modules.

Notice that if $\mu ,\lambda  \in \Lambda $ and $\mu  \sim \lambda $, then there exists $\alpha  \in \mathfrak{S}_r  \subseteq S_r $ such that $G_\mu   = \alpha G_\lambda  \alpha ^{ - 1} $.

\begin{lemma} \label{l5.1}
  If $\mu ,\lambda  \in \Lambda $, $\mu  \sim \lambda $, and $G_\mu   = \alpha G_\lambda  \alpha ^{ - 1} $, then for $X = L{\text{ or }}R$,  $1_\mu   = f\left( {\mu ,\alpha ,\lambda } \right) * _X 1_\lambda   * _X f\left( {\lambda ,\alpha ^{ - 1} ,\mu } \right)$.
\end{lemma}

\begin{proof}
  For any $\rho  \in G_\lambda  $ we have $\alpha \rho \alpha ^{ - 1}  \in G_\mu   = G_\mu  1G_\mu  $, so
\[N\left( {D\left( {\mu ,\alpha ,\lambda } \right),D\left( {\lambda ,\alpha ^{ - 1} ,\mu } \right),D} \right) =
\begin{cases}
o\left({G_\lambda}\right),    &\text{if $D=D\left({\mu,1,\mu}\right)$,}\\
0,       &\text{otherwise.}
\end{cases} \]  
For any $\sigma  \in G_\mu  $, let $\pi  = \alpha ^{ - 1} \sigma \alpha  \in G_\lambda  $.  Then $\sigma \alpha  = \alpha \pi $.  Then $G_R \left( {\mu ,\alpha ,\lambda } \right) = G_\mu  $ and $n_R \left( {\mu ,\alpha ,\lambda } \right) = o\left( {G_\mu  } \right)$.  Reversing the roles of $\lambda ,\mu $ gives $n_R \left( {\lambda ,\alpha ^{ - 1} ,\mu } \right) = o\left( {G_\lambda  } \right)$.  Also $n_R \left( {\mu ,1,\mu } \right) = o\left( {G_\mu  } \right)$.  
Substituting into the multiplication rule then gives
\[f\left( {\mu ,\alpha ,\lambda } \right) * _R 1_\lambda   * _R f\left( {\lambda ,\alpha ^{ - 1} ,\mu } \right) = f\left( {\mu ,\alpha ,\lambda } \right) * _R f\left( {\lambda ,\alpha ^{ - 1} ,\mu } \right)  \]

\[= \frac{{o\left( {G_\mu  } \right)o\left( {G_\lambda  } \right)}}
{{o\left( {G_\mu  } \right)o\left( {G_\lambda  } \right)}}f\left( {\mu ,1,\mu } \right) = 1_\mu.\]
By a similar argument, $n_L \left( {\mu ,\alpha ,\lambda } \right) = o\left( {G_\lambda  } \right)$, $n_L \left( {\lambda ,\alpha ^{ - 1} ,\mu } \right) = o\left( {G_\mu  } \right)$, and $n_L \left( {\mu ,1,\mu } \right) = o\left( {G_\mu  } \right)$.  The multiplication rule again gives 
\[f\left( {\mu ,\alpha ,\lambda } \right) * _L 1_\lambda   * _L f\left( {\lambda ,\alpha ^{ - 1} ,\mu } \right) = f\left( {\mu ,\alpha ,\lambda } \right) * _L f\left( {\lambda ,\alpha ^{ - 1} ,\mu } \right) \]

\[= \frac{{o\left( {G_\mu  } \right)o\left( {G_\lambda  } \right)}}
{{o\left( {G_\mu  } \right)o\left( {G_\lambda  } \right)}}f\left( {\mu ,1,\mu } \right) = 1_\mu  .\]

\end{proof}

\begin{corollary} \label{c5.1}
$1_\lambda  \bmod K^\lambda  $ is a ``full idempotent'' in $Q^\lambda  $, that is, the two sided ideal generated by $1_\lambda  \bmod K^\lambda  $ in $Q^\lambda  $ is all of $Q^\lambda  $.   
\end{corollary}
\begin{proof}

  Modulo $K^\lambda  $ , $\bar e_\lambda   = \sum\limits_{\mu  \in \Lambda ,\mu  \sim \lambda } {1_\mu  } $.   By the lemma, the two-sided ideal generated by $1_\lambda  $ contains each $1_\mu  $, $\mu  \sim \lambda $ .  So the two sided ideal generated by $1_\lambda  \bmod K^\lambda  $ contains the identity $\bar e_\lambda  \bmod K^\lambda  $
in $Q^\lambda  $ and hence all of $Q^\lambda  $.  
\end{proof}

\begin{proposition} \label{p5.1}

  Isomorphism classes of irreducible $C^\lambda  $-modules correspond one to one with isomorphism classes of irreducible $Q^\lambda  $-modules, and hence to isomorphism classes of irreducible $B$-modules at level $\lambda$.
\end{proposition}

\begin{proof}
  By general idempotent theory, irreducible $C^\lambda  $-modules correspond to irreducible $Q^\lambda  $-modules $I$ such that $\left( {1_\lambda  \bmod K^\lambda  } \right)I \ne 0$.  But by the corollary above, we have a full idempotent, and for a full idempotent $\left( {1_\lambda  \bmod K^\lambda  } \right)I = 0 \Rightarrow I = 0$.  So every irreducible $Q^\lambda  $-module corresponds to a $C^\lambda  $-module as claimed.  
\end{proof}

Remark:  Suppose there are $d$ compositions $\mu $ in $\Lambda (r)$ with $\mu  \sim \lambda  \in \Lambda ^ +  (r)$.  It can be shown that $Q^\lambda  $ is isomorphic as an $R$-algebra to the algebra of $d$ by $d$ matrices with entries in $C^\lambda  $.  This leads to an alternative verification of Proposition \ref{p5.1}.

    In the remainder of this paper we will classify the irreducible $B$-modules (in many cases) by classifying the irreducible $C^\lambda$-modules for all partitions $\lambda $.  It will be useful to rewrite $C^\lambda  $ slightly:  since
    \[\bar B^\lambda   = \left( {1_\lambda   + e_\lambda  } \right)B\left( {1_\lambda   + e_\lambda  } \right) = 1_\lambda  B1_\lambda   + e_\lambda  B1_\lambda   + 1_\lambda  Be_\lambda   + e_\lambda  Be_\lambda  \]
     and
     \[e_\lambda  B1_\lambda   + 1_\lambda  Be_\lambda   + e_\lambda  Be_\lambda   \subseteq Be_\lambda  B \cap \bar B^\lambda   = \bar K^\lambda \]
        we have
       \[C^\lambda   = \bar B^\lambda  /\bar K^\lambda   = 1_\lambda  B1_\lambda  /\left( {1_\lambda  B1_\lambda   \cap Be_\lambda  B} \right)\]
       (where $e_\lambda = \sum\limits_{\mu \in \Lambda ,\mu  < \lambda } {1_\mu} $).

\section{Characteristic zero}

In this section we will assume $R$ is a field $k$ of characteristic zero.  We continue to write $B$ for either of the $k$-algebras $LGS_k $ or $RGS_k $ and let $B^\lambda  ,C^\lambda  $ be as in section 5.  Let $\nu  = 1^{(r)} $ be the ``smallest'' partition.

\begin{thm} \label{t6.1}
  For a field $k$ of characteristic zero,  $C^\nu   \cong k[S_r ]$ (the monoid algebra for $S_r $ over $k$) while $C^\lambda   = 0,\,\lambda  \ne \nu $.  The irreducible left $B$-modules are all at level $\nu $ and correspond to irreducible left $k[S_r ]$-modules.  (The irreducible left $k[S_r ]$-modules in turn correspond to irreducible left $k[\mathfrak{S}_i ]$-modules for various $i \leqslant r$.)
\end{thm}

\begin{proof}
  Let $X$ be either $L$ or $R$.  For the smallest partition $\nu $ we have $e_\nu   = 0,\,\,C^\nu   = 1_\nu  B1_\nu  $.  Then $C^\nu  $ is a $k$- vector space with basis $\left\{ {f\left( {\nu ,D,\nu } \right):D \in \,_\nu  M_\nu  } \right\}$.  Since $G_\nu   = \left\{ 1 \right\}$, the $\nu  - \nu $ double cosets $D_\alpha   = G_\nu  \alpha G_\nu   = \alpha $ consist of single elements of $S_r $.  Then $\,_\nu  M_\nu   = S_r $ and $f\left( {\nu ,D_\alpha  ,\nu } \right) \mapsto \alpha $ gives a vector space isomorphism $\phi :C^\nu   \to k[S_r ]$.  By special case 1 of the multiplication law, 	
\[f\left( {\nu ,\alpha ,\nu } \right) * _X f\left( {\nu ,\beta ,\nu } \right) = \frac{{n_X \left( {D\left( {\nu ,\alpha \beta ,\nu } \right)} \right)}}{{n_X \left( {D\left( {\nu ,\alpha ,\nu } \right)} \right)}}f\left( {\nu ,\alpha \beta ,\nu } \right) = f\left( {\nu ,\alpha \beta ,\nu } \right)
\]
(where we used the fact that $n_X (D(\nu ,\gamma ,\nu )) = 1$ for any $\gamma  \in S_r $).  So $\phi $ is an isomorphism of $k$-algebras.  So irreducible left $B$-modules at level $\nu $ correspond to irreducible left $C^\nu   \cong k[S_r ] $-modules as stated.

Now suppose $\lambda  \in \Lambda ^ +  ,\,\lambda  \ne \nu $. To show $C^\lambda   = 0$ it suffices to prove that $1_\lambda  B1_\lambda   \subseteq Be_\lambda  B$, which will be true if $1_\lambda   \in Be_\lambda  B$.  We will show that $1_\lambda   \in B1_\nu  B$.  Then since $1_\nu   = \bar e_\nu   = \bar e_\nu  e_\lambda  $, we will have $1_\lambda   \in B1_\nu  B = B\bar e_\nu  e_\lambda  B \subseteq Be_\lambda  B$ as desired.
By special case 4 of the multiplication rule, 
\[f\left( {\lambda ,1,\nu } \right) * _X 1_\nu   * _X f\left( {\nu ,1,\lambda } \right) = f\left( {\lambda ,1,\nu } \right) * _X f\left( {\nu ,1,\lambda } \right)\]
\[ = \frac{{o\left( {G_\lambda  } \right)}}
{{o\left( {G_\nu  } \right)}}f\left( {\lambda ,1,\lambda } \right) = o\left( {G_\lambda  } \right)1_\lambda  .\]
Since $k$ has characteristic zero, $o\left( {\mathfrak{S}_\lambda  } \right) \ne 0$ in the field $k$ and we have 
\[1_\lambda   = \frac{1}
{{o\left( {\mathfrak{S}_\lambda  } \right)}}f\left( {\lambda ,1,\nu } \right) * _X 1_\nu   * _X f\left( {\nu ,1,\lambda } \right) \in B1_\nu  B\] as claimed, completing the proof that $C^\lambda   = 0$.  
\end{proof}

\section{Characteristic $p$}

     We now consider the case where $R$ is a field $k$ with positive characteristic $p$.

\begin{definition} \label{d7.1}
A partition $\lambda  \in \Lambda ^ +  (r)$ is a $p$-partition if for each $i$, $1 \leqslant i \leqslant n$, either $\lambda _i  = 0$ or $\lambda _i  = p^{k_i } $ for some integer power $k_i  \geqslant 0$ of $p$.
\end{definition}
\begin{thm} \label{t7.1}
If $\lambda $ is not a $p$-partition, then $C^\lambda   = 0$.
\end{thm}
\begin{proof}
Suppose $\lambda $ is not a $p$-partition.  To show $C^\lambda   = 0$ it suffices to show that $Be_\lambda  B \supseteq 1_\lambda  B1_\lambda  $.  This will be true if $1_\lambda   \in Be_\lambda  B = \sum\limits_{\mu  \in \Lambda ,\mu  < \lambda } {B1_\mu  B} $.  So if we can show that $1_\lambda   \in B1_\nu  B$ for some $\nu  \in \Lambda $ with $\nu  < \lambda $ we will be done.  

Since $\lambda $ is not a $p$-partition, $\lambda _a $ is not a power of $p$ for at least one $a$.  Write $\lambda _a  = sp^k  + R$ where $1 \leqslant s < p,\,\,0 \leqslant R < p^k $ (and $R > 0{\text{ if }}s = 1$).  Define a new composition $\nu  \in \Lambda $ by breaking the block $b_\lambda  ^a $ into $s$ blocks of size $p^k $ and one block of size $R$ (if $R > 0$).  That is, we define
\[
   \nu _i  =
	\begin{cases}
    	\lambda _i,    &\text{for $i < a$;}\\
	    p^k,           &\text{for $a + 1 \leqslant i \leqslant a + s$;}\\
	    R,             &\text{for $i = a + s + 1$;}\\
	    \lambda _{i - s - 1},  &\text{for $i > a + s + 1$.}
	\end{cases}
\]
Evidently $\nu  < \lambda $ and we can treat $\mathfrak{S}_\nu  $ as a subgroup of $\mathfrak{S}_\lambda  $.  Then by special case 4 of the multiplication rule,
\[ f\left( {\lambda ,1,\nu } \right) * _X 1_\nu   * _X f\left( {\nu ,1,\lambda } \right) = f\left( {\lambda ,1,\nu } \right) * _X f\left( {\nu ,1,\lambda } \right)\]
\[ = \frac{{o\left( {G_\lambda  } \right)}}
{{o\left( {G_\nu  } \right)}}f\left( {\lambda ,1,\lambda } \right) = \frac{{o\left( {G_\lambda  } \right)}}
{{o\left( {G_\nu  } \right)}}1_\lambda  .\]
So if $c \equiv \frac{{o\left( {G_\lambda  } \right)}}
{{o\left( {G_\nu  } \right)}} \ne 0$ in $k$, then  
\[1_\lambda   = \frac{1}{c} \cdot f\left( {\lambda ,1,\nu } \right) * _X 1_\nu   * _X f\left( {\nu ,1,\lambda } \right) \in B1_\nu  B\]
as desired.  But a little computation shows that
\[c = \frac{{o\left( {\mathfrak{S}_\lambda  } \right)}}{{o\left( {\mathfrak{S}_\nu  } \right)}} = \frac{{\prod\limits_{1 \leqslant i \leqslant r} {o\left( {\mathfrak{S}_{\lambda _i } } \right)} }}
{{\prod\limits_{1 \leqslant i \leqslant r} {o\left( {\mathfrak{S}_{\nu _i } } \right)} }}\]
\[ = \frac{{o\left( {\mathfrak{S}_{\lambda _a } } \right)}}
{{\prod\limits_{a + 1 \leqslant i \leqslant a + s + 1} {o\left( {\mathfrak{S}_{\nu _i } } \right)} }} = \frac{{\left( {sp^k  + R} \right)!}}{{[p^k !]^s  \cdot R!}} \ne 0\,\bmod p.\]  So $c \ne 0$ in $k$, which completes the proof.  
\end{proof}

Now consider a fixed $p$-partition $\lambda $.  Write $s_i  = L(\lambda ,p^i ) = \# \left\{ {j:\lambda _j  = p^i } \right\}$, so there are $s_i $ blocks $b_j ^\lambda  $ of size $p^i $.  We will define a map from the product semigroup $\prod\limits_{s_i  > 0} {\bar \tau _{s_i }}$
to $\bar \tau _r $.  For each $i$ with $s_i  > 0$, write the integers in these $s_i $ blocks in the form $c + (a - 1)p^i  + b$ where $1 \leqslant a \leqslant s_i \,,\,1 \leqslant b \leqslant p^i $ (and $c = \sum\nolimits_{j > i} {s_j p^j } $ is the number of integers in blocks larger than $p^i $).  Then given $\left\{ {\alpha _i  \in \bar \tau _{s_i } :s_i  > 0} \right\}$ define $\alpha  \in \bar \tau _r $ as follows: 
if  $j = c +(a - 1)p^i + b$ is in one of the $s_i$ blocks of size $p^i $, then
\[
\alpha (j) = \alpha \left( {c + (a - 1)p^i  + b} \right) =
     \begin{cases}
       c+\left({\alpha_i (a)-1}\right)p^i+ b,   &\text{if $\alpha _i (a) \ne 0$;}\\
       0,       &\text{if $\alpha _i (a) = 0$}
      \end{cases}
\] 
This defines a map $\phi _\lambda  :\prod\limits_{s_i  > 0} {\bar \tau _{s_i } }  \to \bar \tau _r $ where $\phi _\lambda  \left( {\prod {\alpha _i } } \right) = \alpha $.  It is not hard to check that $\phi  = \phi _\lambda  $ is an injective semigroup homomorphism.  (To understand the map $\phi $, consider the tableau obtained by filling the Young diagram corresponding to $\lambda $ with the integers 1 to $r$ in order from left to right along row 1, then row 2, etc.  Then $\phi \left( {\prod {\alpha _i } } \right) = \alpha $ maps the entries in the $a^{th}$ row of length $p^i $ one to one to the entries in the $\alpha _i \left( a \right)^{th} $ row of length $p^i $ if $\alpha _i (a) \ne 0$ or to 0 if $\alpha _i \left( a \right) = 0$.)

     Next let $S_r ^\lambda   = image\left( {\phi _\lambda  } \right) \cap S_r  \subseteq \bar \tau _r $.  We sometimes identify the semigroup $S_r ^\lambda  $ with its inverse image $\phi _\lambda  ^{ - 1} \left( {S_r ^\lambda  } \right) \subseteq \prod\limits_{s_i  > 0} {\bar \tau _{s_i } } $.  Let $k\left[ {S_r ^\lambda  } \right]$ be the semigroup algebra and (with a slight change in notation) write $B^\lambda   = 1_\lambda  B1_\lambda  $, where $B = B_L {\text{ or }}B_R $.  Then define a $k$-linear map $\psi _\lambda  :k\left[ {S_r ^\lambda  } \right] \to B^\lambda   = 1_\lambda  B1_\lambda  $ by $\psi _\lambda  (\alpha ) = f\left( {\lambda ,\alpha ,\lambda } \right)$ for $\alpha  \in S_r ^\lambda  $ (extended linearly).  Evidently no two distinct elements in $S_r ^\lambda  $ can lie in the same double coset in$\,_\lambda  M_\lambda  $, so $\psi _\lambda  $ is injective.  Write $\psi _{\lambda ,R} $ for the map $\psi _\lambda  $ when $B$ is the right generalized Schur algebra $B_R $ . 

\begin{proposition}  \label{p7.1}
$\psi _{\lambda ,R}  = \psi _\lambda  :k\left[ {S_r ^\lambda  } \right] \to B_R ^\lambda  $ is an (injective) $k$-algebra homomorphism.
\end{proposition}
\begin{proof}
 We must show that for any $\alpha ,\beta  \in S_r ^\lambda  $, $\psi _\lambda  \left( {\alpha \beta } \right) = \psi _\lambda  \left( \alpha  \right)\psi _\lambda  \left( \beta  \right)$.  Observe that for any $\delta  \in S_r ^\lambda  $, we have $\mathfrak{S}_R \left( {\lambda ,\delta ,\lambda } \right) = \mathfrak{S}_\lambda  $ and therefore $n_R \left( {\lambda ,\delta ,\lambda } \right) = o\left( {\mathfrak{S}_\lambda  } \right)$.  Also $\alpha \pi \beta  \in D_{\alpha \beta } $ for any $\alpha ,\beta  \in S_r ^\lambda  \,,\,\pi  \in \mathfrak{S}_\lambda  $ so $N\left( {D_\alpha  ,D_\beta  ,D} \right) = o\left( {\mathfrak{S}_\lambda  } \right){\text{ if }}D = D_{\alpha \beta } $ and equals $0$ otherwise.  The multiplication rule then gives
 \[\psi _\lambda  (\alpha )\psi _\lambda  (\beta ) = f\left( {\lambda ,\alpha ,\lambda } \right) * _R f\left( {\lambda ,\beta ,\lambda } \right) = \frac{{o\left( {G_\lambda  } \right) \cdot n_R (\lambda ,\alpha \beta ,\lambda )}}{{n_R (\lambda ,\alpha ,\lambda )n_R (\lambda ,\beta ,\lambda )}} \cdot f\left( {\lambda ,\alpha \beta ,\lambda } \right).\]
 Then
 \[\psi _\lambda  (\alpha )\psi _\lambda  (\beta ) = \frac{{o\left( {\mathfrak{S}_\lambda  } \right) \cdot o\left( {\mathfrak{S}_\lambda  } \right)}}{{o\left( {\mathfrak{S}_\lambda  } \right) \cdot o\left( {\mathfrak{S}_\lambda  } \right)}} \cdot f\left( {\lambda ,\alpha \beta ,\lambda } \right) = \psi _\lambda  (\alpha \beta )\] as claimed.  
\end{proof}

For the left generalized Schur algebra $B_L $, consider the submonoid $\bar \tau _{s_0 }  \cdot \prod\limits_{i > 0} {\mathfrak{S}_{s_i } }$ of  $\prod\limits_{i \geqslant 0} {\bar \tau _{s_i } } $.  Let $\phi _{\lambda ,L} $ be $\phi _\lambda  $ restricted to $\bar \tau _{s_0 }  \cdot \prod\limits_{i > 0} {\mathfrak{S}_{s_i } } $, $S_{r,L} ^\lambda   = image\left( {\phi _{\lambda ,L} } \right) \cap S_r  \subseteq S_r ^\lambda  $, and  $\psi _{\lambda ,L}  = \psi _\lambda  $ restricted to $k\left[ {S_{r,L} ^\lambda  } \right]$.  Then $\psi _{\lambda ,L} :k\left[ {S_{r,L} ^\lambda  } \right] \to B_L ^\lambda   = 1_\lambda  B_L 1_\lambda  $ is an injective, $k$-linear map.   

\begin{proposition}  \label{p7.2}
$\psi _{\lambda ,L} :k\left[ {S_{r,L} ^\lambda  } \right] \to B_L ^\lambda  $ is an (injective) k-algebra homomorphism.
\end{proposition}
\begin{proof}
  The proof is similar to that of proposition \ref{p7.1}.  We must show that for any $\alpha ,\beta  \in S_{r,L} ^\lambda  $, $\psi _{\lambda ,L} \left( {\alpha \beta } \right) = \psi _{\lambda ,L} \left( \alpha  \right)\psi _{\lambda ,L} \left( \beta  \right)$.  The key is that for the particular elements $\delta  \in S_{r,L} ^\lambda  $, we have $\mathfrak{S}_L \left( {\lambda ,\delta ,\lambda } \right) = \mathfrak{S}_\lambda  $ and therefore $n_L \left( {\lambda ,\delta ,\lambda } \right) = o\left( {\mathfrak{S}_\lambda  } \right)$.  Also we again have $\alpha \pi \beta  \in D_{\alpha \beta } $ for any $\alpha ,\beta  \in S_{r,L} ^\lambda  \,,\,\pi  \in \mathfrak{S}_\lambda  $, so $N\left( {D_\alpha  ,D_\beta  ,D} \right) = o\left( {\mathfrak{S}_\lambda  } \right){\text{ if }}D = D_{\alpha \beta } $ and equals $0$ otherwise.  The multiplication rule then gives
  \[\psi _{\lambda ,L} (\alpha )\psi _{\lambda ,L} (\beta ) = f\left( {\lambda ,\alpha ,\lambda } \right) * _L f\left( {\lambda ,\beta ,\lambda } \right) = \frac{{o\left( {G_\lambda  } \right) \cdot n_L (\lambda ,\alpha \beta ,\lambda )}}
  {{n_L (\lambda ,\alpha ,\lambda )n_L (\lambda ,\beta ,\lambda )}} \cdot f\left( {\lambda ,\alpha \beta ,\lambda } \right).\]
  Then
  \[\psi _{\lambda ,L} (\alpha )\psi _{\lambda ,L} (\beta ) = \frac{{o\left( {\mathfrak{S}_\lambda  } \right) \cdot o\left( {\mathfrak{S}_\lambda  } \right)}}{{o\left( {\mathfrak{S}_\lambda  } \right) \cdot o\left( {\mathfrak{S}_\lambda  } \right)}} \cdot f\left( {\lambda ,\alpha \beta ,\lambda } \right) = \psi _{\lambda ,L} (\alpha \beta )\] as desired.  
\end{proof}

     Now let $\Pi _X :B_X ^\lambda   \to C_X ^\lambda   = B_X ^\lambda  /\left( {B_X e_\lambda  B_X  \cap B_X ^\lambda  } \right)$, for $X = L{\text{ or }}R$, be the natural projection.  We will eventually show that the $k$-algebra homomophisms $\Pi _X  \circ \psi _{\lambda ,X} :k\left[ {S_{r,X} ^\lambda  } \right] \to C_X ^\lambda  $ are surjective.  The following lemma will be used to show that certain $f\left( {\lambda ,\alpha ,\lambda } \right)$ are in $\ker \left( {\Pi _X } \right) = B_X e_\lambda  B_X  \cap B_X ^\lambda  $.

\begin{lemma}  \label{l7.1}
Given a partition $\lambda  \in \Lambda ^ +  $ and $\alpha  \in S_r $, suppose there exists a composition $\nu  \in \Lambda $ such that $G_\nu   \subseteq G_\lambda  $, $\nu  < \lambda $, and either $G_X \left( {\lambda ,\alpha ,\lambda } \right) \subseteq G_X \left( {\nu ,\alpha ,\lambda } \right)$ or $G_X \left( {\lambda ,\alpha ,\lambda } \right) \subseteq G_X \left( {\lambda ,\alpha ,\nu } \right)$.  Then $f\left( {\lambda ,\alpha ,\lambda } \right)  \in \ker \left( {\Pi _X } \right)$.
\end{lemma}
\begin{proof}
It is easy to check that $G_\nu   \subseteq G_\lambda  $ implies $G_X \left( {\lambda ,\alpha ,\lambda } \right) \supseteq G_X \left( {\nu ,\alpha ,\lambda } \right)$ and $G_X \left( {\lambda ,\alpha ,\lambda } \right) \supseteq G_X \left( {\lambda ,\alpha ,\nu } \right)$, so our hypotheses imply either $G_X \left( {\lambda ,\alpha ,\lambda } \right) = G_X \left( {\nu ,\alpha ,\lambda } \right)$ or $G_X \left( {\lambda ,\alpha ,\lambda } \right) = G_X \left( {\lambda ,\alpha ,\nu } \right)$, and hence either $n_X \left( {\lambda ,\alpha ,\lambda } \right) = n_X \left( {\nu ,\alpha ,\lambda } \right)$ or $n_X \left( {\lambda ,\alpha ,\lambda } \right) = n_X \left( {\lambda ,\alpha ,\nu } \right)$.  Suppose $n_X \left( {\lambda ,\alpha ,\lambda } \right) = n_X \left( {\nu ,\alpha ,\lambda } \right)$.  Then case 2 of the multiplication rule gives
 \[f\left( {\lambda ,1,\nu } \right) * _X f\left( {\nu ,\alpha ,\lambda } \right) = \frac{{n_X \left( {\lambda ,\alpha ,\lambda } \right)}}
{{n_X \left( {\nu ,\alpha ,\lambda } \right)}}f\left( {\lambda ,\alpha ,\lambda } \right) = f (\lambda , \alpha , \lambda ).\]  Since $\nu  < \lambda $ we have $e_\lambda  1_\nu   = 1_\nu  $, so
 \[f\left( {\lambda ,1,\nu } \right) * _X f\left( {\nu ,\alpha ,\lambda } \right) = f\left( {\lambda ,1,\nu } \right) * _X 1_\nu   * _X f\left( {\nu ,\alpha ,\lambda } \right) = \]

\[f\left( {\lambda ,1,\nu } \right) * _X e_\lambda  1_\nu   * _X f\left( {\nu ,\alpha ,\lambda } \right) \in B_X e_\lambda  B_X  \cap B_X ^\lambda   = \ker \left( {\Pi _X } \right) .\]
Then $f\left( {\lambda ,\alpha ,\lambda } \right)  \in \ker \left( {\Pi _X } \right)$ as desired.
     If $n_X \left( {\lambda ,\alpha ,\lambda } \right) = n_X \left( {\lambda ,\alpha ,\nu } \right)$, a parallel proof using case 3 of the multiplication rule again gives $f\left( {\lambda ,\alpha ,\lambda } \right)  \in \ker \left( {\Pi _X } \right)$.  
\end{proof}

     Let $\lambda  \in \Lambda $ be a composition and take any $\alpha  \in S_r $.  Consider $\alpha $ as a matrix with rows $\alpha _i $ indexed by $i = 0,1,2, \cdots ,r$. 

\begin{definition}  \label{d7.2}
Two rows $\alpha _i ,\alpha _j $ of $\alpha $ are $\lambda $-equivalent if $\# \left( {\alpha ^{ - 1} (i) \cap b_k ^\lambda  } \right) = \# \left( {\alpha ^{ - 1} (j) \cap b_k ^\lambda  } \right)$ for every $\lambda $-block $b_k ^\lambda  $.
\end{definition}
  This means that the two rows have the same number of 1's in columns belonging to any $\lambda $-block.  If we let $\mathfrak{S}_r $ act (on the right) on the rows of $\alpha $ by permuting elements, then $\alpha _i ,\alpha _j $ are $\lambda $-equivalent if and only if $\alpha _j  = \alpha _i \sigma $ for some $\sigma  \in \mathfrak{S}_\lambda  $.  

\begin{lemma}  \label{l7.2}
For any $\lambda  \in \Lambda $ and $\alpha  \in S_r $, suppose there exist $\pi  \in G_\lambda  ,\sigma  \in \mathfrak{S}_r $, such that $\alpha \pi  = \sigma \alpha $.  Then for any $1 \leqslant i \leqslant r$, rows $i$ and $\sigma \left( i \right)$ of $\alpha $ are $\lambda $-equivalent.
\end{lemma}
\begin{proof}
For a given $\lambda $-block $b_k ^\lambda  $, assume $\# \left( {\alpha ^{ - 1} (i) \cap b_k ^\lambda  } \right) = m$ and write $\alpha ^{ - 1} (i) \cap b_k ^\lambda   = \left\{ {c_1 ,c_2 , \cdots ,c_m } \right\}$.  Then $\alpha \left( {c_j } \right) = i$, so $\sigma \left( i \right) = \sigma \alpha \left( {c_j } \right) = \alpha \pi \left( {c_j } \right)$.  Then $\pi \left( {c_j } \right) \in \alpha ^{ - 1} \left( {\sigma \left( i \right)} \right)$.  Also, $\pi \left( {c_j } \right) \in b_k ^\lambda  $ since $\pi  \in G_\lambda  $.  So $\alpha ^{ - 1} (\sigma (i)) \cap b_k ^\lambda   = \left\{ {\pi c_1 ,\pi c_2 , \cdots ,\pi c_m } \right\}$.  Then $\# \left( {\alpha ^{ - 1} (\sigma (i)) \cap b_k ^\lambda  } \right) = m = \# \left( {\alpha ^{ - 1} (i) \cap b_k ^\lambda  } \right)$.  Since this is true for all $\lambda $-blocks, rows $i$ and $\sigma \left( i \right)$ of $\alpha $ are $\lambda $-equivalent as claimed.  
\end{proof}
 
\begin{lemma}  \label{l7.3}
If two rows $\alpha _i ,\alpha _j $ of $\alpha $ with $i,j$ in the same $\lambda $-block $b_k ^\lambda  $ are not $\lambda $-equivalent, then $f\left( {\lambda ,\alpha ,\lambda } \right)  \in \ker \left( {\Pi _X } \right)$.
\end{lemma}
\begin{proof}
Suppose $b_k ^\lambda   = \left\{ {c + 1,c + 2, \cdots ,c + \lambda _k } \right\}$ is a $\lambda $-block for which two rows of $\alpha $ are not $\lambda $-equivalent.  By choosing, if necessary, a different representative for $\alpha $ in the same $G_\lambda   - G_\lambda  $
 double coset, we can assume that rows $c + 1,c + 2, \cdots ,c + a$ are all $\lambda $-equivalent for some $1 \leqslant a < \lambda _k $, while none of the rows $c + a + 1,c + a + 2, \cdots ,c + \lambda _k $ are $\lambda $-equivalent to row $c + 1$.  Let $\nu $ be the composition obtained by splitting block $b_k ^\lambda  $ into two (nonempty) blocks $\left\{ {c + 1, \cdots ,c + a} \right\}$ and $\left\{ {c + a + 1, \cdots ,c + \lambda _k } \right\}$, while leaving the other blocks unchanged.  Then $G_\nu   \subseteq G_\lambda  $ and $\nu  < \lambda $.  By lemma \ref{l7.1}, if we show $G_X \left( {\lambda ,\alpha ,\lambda } \right) \subseteq G_X \left( {\nu ,\alpha ,\lambda } \right)$ then $f\left( {\lambda ,\alpha ,\lambda } \right)  \in \ker \left( {\Pi _X } \right)$ and the proof is complete.

    For $X = R$, take any $\sigma  \in G_R \left( {\lambda ,\alpha ,\lambda } \right) \subseteq \mathfrak{S}_\lambda  $.  Then there exists a $\pi  \in G_\lambda  $ such that $\alpha \pi  = \sigma \alpha $.  By lemma \ref{l7.2}, rows $i$ and $\sigma \left( i \right)$ of $\alpha $
 must be $\lambda $-equivalent for any $i$.  Then in particular, $\sigma $ must map blocks  $\left\{ {c + 1, \cdots ,c + a} \right\}$ and $\left\{ {c + a + 1, \cdots ,c + \lambda _k } \right\}$ into themselves. Since $\sigma  \in G_\lambda  $, it maps all other blocks to themselves.  Then $\sigma  \in G_\nu  $ and $\sigma  \in G_\nu   \cap G_R \left( {\lambda ,\alpha ,\lambda } \right) = G_R \left( {\nu ,\alpha ,\lambda } \right)$.  It follows that $G_R \left( {\lambda ,\alpha ,\lambda } \right) \subseteq G_R \left( {\nu ,\alpha ,\lambda } \right)$ as desired.

The proof for $X = L$ is similar. Take any $\pi  \in G_L \left( {\lambda ,\alpha ,\lambda } \right) \subseteq \mathfrak{S}_\lambda  $.  Then there exists a $\sigma  \in G_\lambda  $ such that $\alpha \pi  = \sigma \alpha $.  Then lemma \ref{l7.2} again shows that rows $i$ and $\sigma \left( i \right)$ of $\alpha $ must be $\lambda $-equivalent for any $i$.  Then $\sigma  \in G_\nu  $ (as for the case $X = R$ above), which implies $\pi  \in G_L \left( {\nu ,\alpha ,\lambda } \right)$.  It follows that $G_L \left( {\lambda ,\alpha ,\lambda } \right) \subseteq G_L \left( {\nu ,\alpha ,\lambda } \right)$ as desired.  
\end{proof}

\begin{corollary}  \label{c7.1}
Take $\alpha  \in S_r $ with $f\left( {\lambda ,\alpha ,\lambda } \right)  \notin \ker \left( {\Pi _X } \right)$.  Suppose $i \in b_k ^\lambda  $, a block of size $\lambda _k $, and $\alpha \left( i \right) \in b_j ^\lambda  $, a block of size $\lambda _j $.  Then $\lambda _k  \geqslant \lambda _j $.  If $\lambda _k  = \lambda _j $, then $\alpha $ maps the block $b_k ^\lambda  $ one to one onto the block $b_j ^\lambda  $.
\end{corollary}
\begin{proof}
By lemma \ref{l7.3}, any row $\alpha _l $ with $l \in b_j ^\lambda  $
 must be $\lambda $-equivalent to row $\alpha \left( i \right)$.  Then since $i \in \alpha ^{ - 1} (\alpha \left( i \right)) \cap b_k ^\lambda  $, each set $\alpha ^{ - 1} \left( {\alpha _l } \right) \cap b_k ^\lambda  $ for $l \in b_j ^\lambda  $ must contain at least one element.  These $\lambda _j $ nonempty sets are disjoint subsets of the set $b_k ^\lambda  $ which has $\lambda _k $ elements, so clearly $\lambda _k  \geqslant \lambda _j $.  If $\lambda _k  = \lambda _j $, then each set $\alpha ^{ - 1} \left( {\alpha _l } \right) \cap b_k ^\lambda  $ must contain exactly one element and every element of $b_k ^\lambda  $ must lie in one such set.  But then $\alpha $ maps the block $b_k ^\lambda  $ one to one onto the block $b_j ^\lambda  $ as claimed.  
\end{proof}

We now treat $B_R $ and $B_L $ separately.  For $B_R $ we have

\begin{lemma}    \label{l7.4}
For $\alpha  \in S_r $, suppose there exists $i \in b_k ^\lambda  $, a block of size $\lambda _k $, with $\alpha \left( i \right) \in b_j ^\lambda  $, a block of size $\lambda _j $ where $\lambda _k  > \lambda _j $.  Then $f\left( {\lambda ,\alpha ,\lambda } \right) \in \ker \left( {\Pi _R } \right)$.  
\end{lemma}
\begin{proof}
By lemma \ref{l7.3}, $f\left( {\lambda ,\alpha ,\lambda } \right) \in \ker \left( {\Pi _R } \right)$ if there are two rows in the block $b_j ^\lambda  $ which are not $\lambda $-equivalent, so assume all rows in block $b_j ^\lambda  $ are $\lambda $-equivalent.  Then for all $l \in b_j ^\lambda  $, the disjoint sets $\alpha ^{ - 1} \left( l \right) \cap b_k ^\lambda  $ contain the same number, say $m$, of elements.  Then by altering $\alpha $ within its double coset if necessary, we can assume that the block $b_k ^\lambda $ has the form
\[
c + 1, \cdots ,c + \lambda _j ,c + \lambda _j  + 1, \cdots ,c + 2\lambda _j , \cdots ,\]
\[c + (m - 1)\lambda _j  + 1, \cdots ,c + m\lambda _j ,c + m\lambda _j  + 1, \cdots ,c + \lambda _k 
\]
where each sub-block $\left\{ {c + i\lambda _j  + 1, \cdots ,c + (i + 1)\lambda _j } \right\}$ is mapped by $\alpha $ one to one onto block $b_j ^\lambda  $, while none of $\left\{ {c + m\lambda _j  + 1, \cdots ,c + \lambda _k } \right\}$ is mapped to block $b_j ^\lambda  $.

Let $\nu $ be the composition obtained by splitting block $b_k ^\lambda  $ into two (nonempty) blocks $\left\{ {c + 1, \cdots ,c + \lambda _j } \right\}$ and $\left\{ {c + \lambda _j  + 1, \cdots ,c + \lambda _k } \right\}$, while leaving the other blocks unchanged.  Then $G_\nu   \subseteq G_\lambda  $ and $\nu  < \lambda $.  By lemma \ref{l7.1}, if we show $G_R \left( {\lambda ,\alpha ,\lambda } \right) \subseteq G_R \left( {\lambda ,\alpha ,\nu } \right)$ then $f\left( {\lambda ,\alpha ,\lambda } \right) \in \ker \left( {\Pi _R } \right)$ and the proof is complete.

     Take any $\sigma  \in G_R \left( {\lambda ,\alpha ,\lambda } \right)$.  There exists $\pi  \in G_\lambda  $ such that $\sigma \alpha  = \alpha \pi $.  To show that $\sigma  \in G_R \left( {\lambda ,\alpha ,\nu } \right)$ we must find a $\pi ' \in G_\nu  $ such that $\sigma \alpha  = \alpha \pi '$.  Let $C = \left\{ {c + 1, \cdots ,c + m\lambda _j } \right\} \subseteq b_k ^\lambda  $.  Notice that $\pi $ must map $C$ onto $C$ and $b_k ^\lambda   - C$ onto $b_k ^\lambda   - C$:  if $\pi \left( i \right) \in b_k ^\lambda   - C$ for some $i \in C$, then $\alpha \pi \left( i \right) \notin b_j ^\lambda  $, while $\sigma \alpha \left( i \right) \in b_j ^\lambda  $, contradicting $\sigma \alpha  = \alpha \pi $.   Then if $\pi '\left( i \right) = \pi \left( i \right)$ for $i \notin C$, $\pi '$ maps each $\lambda $-block not equal to $b_k ^\lambda  $ onto itself and also maps $b_k ^\lambda   - C$ onto itself.  For each sub-block $C_i  = \left\{ {c + i\lambda _j  + 1, \cdots ,c + i\lambda _j  + \lambda _j } \right\}\,,\,i = 0,1, \cdots ,m - 1$, both $\alpha $ and $\sigma \alpha $ map $C_i $ one to one onto $b_j ^\lambda  $.  Then define $\pi '$ restricted to $C_i $ to be the permutation $\left. {\alpha ^{ - 1} } \right|b_j ^\lambda   \circ \left. {\sigma \alpha } \right|C_i $ of $C_i $ such that $\left. {\alpha \pi '} \right|C_i  = \left. {\sigma \alpha } \right|C_i $.  Then $\pi ' \in G_\nu  $ and $\sigma \alpha  = \alpha \pi '$ , so $\sigma  \in G_R \left( {\lambda ,\alpha ,\nu } \right)$.  So $G_R \left( {\lambda ,\alpha ,\lambda } \right) \subseteq G_R \left( {\lambda ,\alpha ,\nu } \right)$ and we are done.  
\end{proof}

\begin{proposition}   \label{p7.3}
If $D_\alpha   \notin \left\{ {D_{\alpha '} :\alpha ' \in S_r ^\lambda  } \right\}$, then $f\left( {\lambda ,\alpha ,\lambda } \right) \in \ker \left( {\Pi _R } \right)$.  The $k$-algebra homomorphism $\Pi _R  \circ \psi _{\lambda ,R} :k\left[ {S_r ^\lambda  } \right] \to C_R ^\lambda  $ is surjective.
\end{proposition}
\begin{proof}
Take any $\alpha  \in S_r $ with $f\left( {\lambda ,\alpha ,\lambda } \right)  \notin \ker \left( {\Pi _R } \right)$ and consider a $\lambda $-block $b_k ^\lambda  $ of size $\lambda _k $.  By corollary \ref{c7.1}, $\alpha $ cannot map any element of $b_k ^\lambda  $ into a block of size greater than $\lambda _k $ and, by lemma \ref{l7.4},  $\alpha $ cannot map any element of $b_k ^\lambda  $ into a block of size less than $\lambda _k $.  So, by corollary \ref{c7.1} again, either $\alpha $ maps all of $b_k ^\lambda  $ to 0 or it maps $b_k ^\lambda  $ one to one onto another block of size $\lambda _k $.  But then $\alpha $ is in the same $G_\lambda   - G_\lambda  $ double coset as some $\alpha ' \in S_r ^\lambda  $, that is, $D_\alpha   = D_{\alpha '} $
 for $\alpha ' \in S_{r,R} ^\lambda  $.

     Since the projection $\Pi _R :B_R ^\lambda   \to C_R ^\lambda  $ is surjective and we have just shown that any basis element $D_\alpha  $ for $B_R ^\lambda  $ not in $\left\{ {D_{\alpha '} :\alpha ' \in S_r ^\lambda  } \right\}$ is in $\ker \left( {\Pi _R } \right)$, $\Pi _R $ must map the $k$-span of $\left\{ {D_{\alpha '} :\alpha ' \in S_r ^\lambda  } \right\}$ onto $C_R ^\lambda  $.  But the $k$-span of $\left\{ {D_{\alpha '} :\alpha ' \in S_r ^\lambda  } \right\}$ is exactly $\psi _{\lambda ,R} \left( {k\left[ {S_r ^\lambda  } \right]} \right)$, so  $\Pi _R  \circ \psi _{\lambda ,R} :k\left[ {S_r ^\lambda  } \right] \to C_R ^\lambda  $
 is surjective.  
\end{proof}

     For $B_L $ we have
\begin{lemma}   \label{l7.5}
For any $\lambda  \in \Lambda ^ +  $ and $\alpha  \in S_r $, if $b_k ^\lambda  $ is a block of size $\lambda _k  > 1$ and $\alpha ^{ - 1} \left( {b_k ^\lambda  } \right) = \emptyset $, then $f\left( {\lambda ,\alpha ,\lambda } \right)  \in \ker \left( {\Pi _L } \right)$.
\end{lemma}

\begin{proof}
  If $\alpha ^{ - 1} \left( {b_k ^\lambda  } \right) = \emptyset $, Let $\nu $ be the composition obtained by breaking block $b_k ^\lambda  $ into $\lambda _k $ blocks of size $1$.  Then $G_\nu \subseteq G_\lambda $
and $\nu  < \lambda $ (assuming $\lambda _k  > 1$).   By lemma 6.1, if we show $G_L \left( {\lambda ,\alpha ,\lambda } \right) \subseteq G_L \left( {\nu ,\alpha ,\lambda } \right)$ then $f\left( {\lambda ,\alpha ,\lambda } \right) \in \ker \left( {\Pi _L } \right)$ and the proof is complete.

     Take any $\pi  \in G_L \left( {\lambda ,\alpha ,\lambda } \right)$ and a corresponding $\sigma  \in G_\lambda  $ such that $\alpha \pi  = \sigma \alpha $.  Define $\sigma ' \in G_\lambda  $ by $\sigma '\left( i \right) = \left\{ {\begin{array}{*{20}c}
   {i{\text{ if }}i \in b_k ^\lambda  }  \\
   {\sigma \left( i \right){\text{ otherwise}}}  \\

 \end{array} } \right.$.  Then $\sigma ' \in G_\nu  $.  But since, for all $j$, $\alpha \left( j \right) \notin b_k ^\lambda  $, we have $\sigma '\alpha \left( j \right) = \sigma \alpha \left( j \right) = \alpha \pi \left( j \right)$ for all $j$.  Then $\sigma '\alpha  = \alpha \pi $ and  $\pi  \in G_L \left( {\nu ,\alpha ,\lambda } \right)$.  So  $G_L \left( {\lambda ,\alpha ,\lambda } \right) \subseteq G_L \left( {\nu ,\alpha ,\lambda } \right)$.  
\end{proof}
\begin{corollary}    \label{c7.2}
 Suppose $f\left( {\lambda ,\alpha ,\lambda } \right)  \notin \ker \left( {\Pi _L } \right)$.  Then for any block $b_k ^\lambda  $ of size $\lambda _k $, $\alpha $ maps $b_k ^\lambda  $ one to one onto a block $b_j ^\lambda  $ of the same size $\lambda _j  = \lambda _k $. Furthermore, for any block $b_k ^\lambda  $ of size $\lambda _k  > 1$,  $\alpha ^{ - 1} \left( {b_k ^\lambda  } \right) = b_j ^\lambda  $ for some block $b_j ^\lambda  $ of the same size $\lambda _j  = \lambda _k $.  
\end{corollary}
\begin{proof}
Use ``downward induction'':  Assume for some size $s$ that $\alpha $
maps all blocks of size $ > s$ one to one onto blocks of the same size and that the inverse image of any such block is another block of the same size.  Then consider a block $b_j ^\lambda  $ of size $\lambda _j  = s$.  By our assumption, $\alpha ^{ - 1} \left( {b_j ^\lambda  } \right)$ can contain no element of a block of size greater than $s$.  But by corollary \ref{c7.1}, $\alpha ^{ - 1} \left( {b_j ^\lambda  } \right)$ can contain no element of a block of size less than $s$.  So $\alpha ^{ - 1} \left( {b_j ^\lambda  } \right)$ contains only elements in blocks of size $s$.  By corollary \ref{c7.1} again, if it contains any element of such a block $b_k ^\lambda  $, then  $\alpha $ maps all of $b_k ^\lambda  $ one to one onto $b_j ^\lambda  $.  If $s > 1$ then by lemma \ref{l7.5},  $\alpha ^{ - 1} \left( {b_j ^\lambda  } \right) \ne \emptyset $, so $\alpha ^{ - 1} \left( {b_j ^\lambda  } \right)$ must be a union of one or more blocks of size $s$.  But then since the sets $\alpha ^{ - 1} \left( {b_j ^\lambda  } \right)$ are disjoint for the different blocks $b_j ^\lambda  $
 of size $\lambda _j  = s$, each $\alpha ^{ - 1} \left( {b_j ^\lambda  } \right)$ must consist of a single block and each block $b_k ^\lambda  $ of size $s$ must appear as $\alpha ^{ - 1} \left( {b_j ^\lambda  } \right)$ for exactly one block $b_j ^\lambda  $ of size $\lambda _j  = s$.  By induction, the result is true for all sizes $s$. 
\end{proof}
\begin{proposition}    \label{p7.4}
  If $D_\alpha   \notin \left\{ {D_{\alpha '} :\alpha ' \in S_{r,L} ^\lambda  } \right\}$, then $f\left( {\lambda ,\alpha ,\lambda } \right)  \in \ker \left( {\Pi _L } \right)$.  The $k$-algebra homomorphism $\Pi _L  \circ \psi _{\lambda ,L} :k\left[ {S_{r,L} ^\lambda  } \right] \to C_L ^\lambda  $ is surjective.
\end{proposition}
\begin{proof}
  Take any $\alpha  \in S_r $ with $f\left( {\lambda ,\alpha ,\lambda } \right) \notin \ker \left( {\Pi _L } \right)$.  By corollary \ref{c7.2}, for any $i > 0$ there is a permutation $\sigma  \in \mathfrak{S}_{s_i } $ of the blocks $b_k ^\lambda  $ of size $\lambda _k  = p^i  > 1$ such that $\alpha $ maps the elements of $b_k ^\lambda  $ one to one onto the block $\sigma \left( {b_k ^\lambda  } \right)$ of the same size $p^i $.  Also the blocks of size $1$ are mapped by $\alpha $ to blocks of size one (or to zero).  But then $\alpha $ is in the same $G_\lambda   - G_\lambda  $ double coset as some $\alpha ' \in S_{r,L} ^\lambda  $, that is, $D_\alpha   = D_{\alpha '} $ for $\alpha ' \in S_{r,L} ^\lambda  $.

     Since the projection $\Pi _L :B_L ^\lambda   \to C_L ^\lambda  $ is surjective and we have just shown that any basis element $D_\alpha  $ for $B_L ^\lambda  $ not in $\left\{ {D_{\alpha '} :\alpha ' \in S_{r,L} ^\lambda  } \right\}$ is in $\ker \left( {\Pi _L } \right)$, $\Pi _L $ must map the $k$-span of $\left\{ {D_{\alpha '} :\alpha ' \in S_{r,L} ^\lambda  } \right\}$ onto $C_L ^\lambda  $.  But the $k$-span of $\left\{ {D_{\alpha '} :\alpha ' \in S_{r,L} ^\lambda  } \right\}$ is exactly $\psi _{\lambda ,L} \left( {k\left[ {S_{r,L} ^\lambda  } \right]} \right)$, so  $\Pi _L  \circ \psi _{\lambda ,L} :k\left[ {S_{r,L} ^\lambda  } \right] \to C_L ^\lambda  $ is surjective.  
\end{proof}

  In the next section we will show that $\Pi _L  \circ \psi _{\lambda ,L} :k\left[ {S_{r,L} ^\lambda  } \right] \to C_L ^\lambda  $ is actually an isomorphism of algebras, while in many cases $\Pi _R  \circ \psi _{\lambda ,R} $ restricted to an appropriate subalgebra of $k\left[ {S_{r,R} ^\lambda  } \right]$ is an algebra isomorphism onto $C_R ^\lambda  $.  We can then index the irreducible representations of each $C_X ^\lambda  $ and hence of $B_X $.

\section{Irreducible representations of $C_X ^\lambda$}

We need some technical lemmas to help determine the kernel of $\Pi _X $.  Work first with $X = L$.  For any partition $\lambda  \in \Lambda ^ + $, we can regard $G_{\lambda _i }  = \mathfrak{S}_{\lambda _i } $ as a subgroup of $G_\lambda   = \mathfrak{S}_\lambda  $ by letting $\sigma  \in G_{\lambda _i } $ act as the identity on all blocks $b_j ^\lambda  \,,\,j \ne i$, while $\sigma \left( {a_k } \right) = a_{\sigma \left( k \right)} $ if $a_k $ represents the $k$th element in the block $b_i ^\lambda  $.  Then $G_\lambda  $ is a direct product of disjoint subgroups $G_\lambda   = \prod\limits_i {G_{\lambda _i } } $.  Let $\nu  \in \Lambda $ be a composition and let $\alpha  \in S_{r,L} ^\lambda$.  Put $G_{L,i} \left( {\nu ,\alpha ,\lambda } \right) = \left\{ {\pi  \in G_{\lambda _i } :\exists \sigma  \in G_\nu  {\text{ such that }}\alpha \pi  = \sigma \alpha } \right\}$.  This gives a family of disjoint subgroups of $G_L \left( {\nu ,\alpha ,\lambda } \right)$.

\begin{lemma}   \label{l8.1}
 Let $B_1 $ be the union of all $\lambda $-blocks of size $1$ and $B_2 $
 be the union of all $\lambda $-blocks of size greater than $1$.  Suppose $\alpha  \in S_r $ restricted to $B_2 $ is one to one and that $\alpha \left( {B_1 } \right) \cap \alpha \left( {B_2 } \right) = \emptyset $.  Then $G_L \left( {\nu ,\alpha ,\lambda } \right) = \mathop \Pi \limits_i G_{L,i} \left( {\nu ,\alpha ,\lambda } \right)$.
\end{lemma}

\begin{proof}
 Given $\pi  \in G_L \left( {\nu ,\alpha ,\lambda } \right) \subseteq G_\lambda  $, we must write $\pi  = \mathop \Pi \limits_i \pi _i $ where $\pi _i  \in G_{L,i} \left( {\nu ,\alpha ,\lambda } \right) \subseteq G_{\lambda _i } $.  Define $\pi _i \left( j \right) = \left\{ {\begin{array}{*{20}c}
   {\pi \left( j \right){\text{ if }}j \in b_i ^\lambda  }  \\
   {j{\text{ otherwise}}}  \\

 \end{array} } \right.$.  Then $\pi _i  \in G_{\lambda _i } $ and $\pi  = \mathop \Pi \limits_i \pi _i $.  Take $\sigma  \in G_\nu  $ such that $\alpha \pi  = \sigma \alpha $.  By hypothesis, $\alpha $ maps $b_i ^\lambda  $ one to one onto an image set $S_i $ of the same size and if the size $\lambda _i  > 1$ then $\alpha ^{ - 1} \left( {S_i } \right) = b_i ^\lambda  $.  We claim $\sigma $ maps each image set $S_i $ to itself:  take any $k = \alpha \left( l \right) \in S_i $ where  $l \in b_i ^\lambda  $.  Then $\sigma \left( k \right) = \sigma \alpha \left( l \right) = \alpha \pi \left( l \right) \in S_i $ since $\pi  \in G_\lambda  \Rightarrow \pi \left( l \right) \in b_i ^\lambda  $.  So we can define a permutation $\sigma _i  \in \mathfrak{S}_r $ by $\sigma _i (k) = \left\{ {\begin{array}{*{20}c}   {\sigma \left( k \right){\text{ if }}k \in S_i }  \\   {k{\text{ otherwise}}}  \\
 \end{array} } \right.$.  $\sigma _i $ is the identity outside $S_i $ and therefore certainly maps any portion of a $\nu $ block outside $S_i $ to itself.  On the other hand, when restricted to $S_i $, $\sigma _i  = \sigma  \in G_\nu  $, so it must map any portion of a $\nu $ block inside $S_i $ to itself.  So $\sigma _i  \in G_\nu  $.  Then since $\alpha \pi _i  = \sigma _i \alpha $, we have $\pi _i  \in G_{L,i} \left( {\nu ,\alpha ,\lambda } \right)$.  When the size $\lambda _i  = 1$, $\pi _i $ is the identity map.  If we take $\sigma _i $ to also be the identity, then we again have $\alpha \pi _i  = \sigma _i \alpha $ and $\pi _i  \in G_{L,i} \left( {\nu ,\alpha ,\lambda } \right)$.
\end{proof}

Notice that in particular that any $\alpha  \in S_{r,L} ^\lambda  $ satisfies the hypotheses of lemma \ref{l8.1}.  Also $G_L \left( {\lambda ,\alpha ,\lambda } \right) = G_\lambda  {\text{ and }}G_{L,i} \left( {\lambda ,\alpha ,\lambda } \right) = G_{\lambda _i } $ for $\alpha  \in S_{r,L} ^\lambda  $.

\begin{lemma}    \label{l8.2}
  If an element $x \in \ker \left( {\Pi _L } \right) \subseteq B_L ^\lambda  $ is expanded in terms of the double coset basis $\{ f\left( {\lambda ,D,\lambda } \right):D \in \,_\lambda  M_\lambda  \} $ for $B_L ^\lambda  $, $x = \sum\limits_{D_\alpha   \in \,_\lambda  M_\lambda  } {c_\alpha  f\left( {\lambda ,\alpha ,\lambda } \right)} $, then the coefficient $c_\gamma  $ for any $f\left( {\lambda ,\gamma ,\lambda } \right)$ with $\gamma  \in S_{r,L} ^\lambda  $ will be 0.
\end{lemma}

\begin{proof}
   Since $x \in \ker \left( {\Pi _L } \right) \subseteq B_L ^\lambda  e_\lambda  B_L ^\lambda  $, $x$ will be a $k$ linear combination of terms of the form  $f\left( {\lambda ,\alpha ,\nu } \right) * _L f\left( {\nu ,\beta ,\lambda } \right)$ for some $\nu $ with $\nu  < \lambda $.  Then by the alternative form of the multiplication rule in section 3, $f\left( {\lambda ,\alpha ,\nu } \right) * _L f\left( {\nu ,\beta ,\lambda } \right) = \sum\limits_{D \in \,_\lambda  M_\lambda  } {\frac{{n_L (D)\,\,a_{L,D} }}
{{n_L \left( {\nu ,\beta ,\lambda } \right)}}f\left( {\lambda ,D,\lambda } \right)} $ for certain integers $a_{L,D}  \in k$.  Then the coefficient of any basis element $f\left( {\lambda ,\gamma ,\lambda } \right)$ in the expansion of $x$ will be a $k$ linear combination of terms $\frac{{n_L \left( {\lambda ,\gamma ,\lambda } \right)}}
{{n_L \left( {\nu ,\beta ,\lambda } \right)}}$.  We will show that $\frac{{n_L \left( {\lambda ,\gamma ,\lambda } \right)}}
{{n_L \left( {\nu ,\beta ,\lambda } \right)}} \equiv 0\,\,\bmod (p)$
 whenever $\gamma  \in S_{r,L} ^\lambda  $ and $a_{L,D_\gamma  }  \ne 0$.  Then $c_\gamma   = 0$ in $k$ as desired.

   If $a_{L,D_\gamma  }  \ne 0$, then $\alpha \rho \beta  = \gamma $ for some $\rho  \in G_\nu  $.  Then if $\gamma  \in S_{r,L} ^\lambda  $ , $\beta $ restricted to $B_2 $ will be one to one and $\beta \left( {B_1 } \right) \cap \beta \left( {B_2 } \right) = \emptyset $ where $B_1 ,B_2 $ are as in lemma \ref{l8.1}.  Then by lemma \ref{l8.1} we have  $G_L \left( {\nu ,\beta ,\lambda } \right) = \mathop \Pi \limits_i G_{L,i} \left( {\nu ,\beta ,\lambda } \right)$.  As remarked after lemma \ref{l8.1}, we also have $G_L \left( {\lambda ,\gamma ,\lambda } \right) = G_\lambda  {\text{ and }}G_{L,i} \left( {\lambda ,\gamma ,\lambda } \right) = G_{\lambda _i } $ since $\gamma  \in S_{r,L} ^\lambda  $.  Let $n_{L,i} \left( {\lambda ,\gamma ,\lambda } \right) = o\left( {G_{L,i} \left( {\lambda ,\gamma ,\lambda } \right)} \right) = o\left( {G_{\lambda _i } } \right)$ and $n_{L,i} \left( {\nu ,\beta ,\lambda } \right) = o\left( {G_{L,i} \left( {\nu ,\beta ,\lambda } \right)} \right)$.  Since $G_{L,i} \left( {\nu ,\beta ,\lambda } \right)$ is a subgroup of $G_{L,i} \left( {\lambda ,\gamma ,\lambda } \right) = G_{\lambda _i } $, $n_{L,i} \left( {\nu ,\beta ,\lambda } \right)$ divides $n_{L,i} \left( {\lambda ,\gamma ,\lambda } \right)$.  Then we can write  $\frac{{n_L \left( {\lambda ,\gamma ,\lambda } \right)}}{{n_L \left( {\nu ,\beta ,\lambda } \right)}} = \mathop \Pi \limits_i \frac{{n_{L,i} \left( {\lambda ,\gamma ,\lambda } \right)}}
{{n_{L,i} \left( {\nu ,\beta ,\lambda } \right)}}$.  If we can show $n_i  \equiv \frac{{n_{L,i} \left( {\lambda ,\gamma ,\lambda } \right)}}
{{n_{L,i} \left( {\nu ,\beta ,\lambda } \right)}} \equiv 0\,\,\bmod \left( p \right)$ for at least one $i$, the proof will be complete.

     Since $\beta $ is one to one on $B_2 $ and $\nu  < \lambda $, there must be some $\lambda $-block $b_k ^\lambda  $ of size $\lambda _k  = p^t  > 1$ and an element $i \in b_k ^\lambda  $ such that $\beta \left( i \right) \in b_j ^\nu  $ where the size $\nu _j  < \lambda _k $.  Let $A_1  = b_k ^\lambda   \cap \beta ^{ - 1} \left( {b_j ^\nu  } \right)$, $A_2  = b_k ^\lambda   - A_1 $, $a_i  = \# A_i $.  Then $1 \leqslant a_1  \leqslant \nu _j  < \lambda _k  = p^t $ and $a_1  + a_2  = \# b_k ^\lambda   = \lambda _k  = p^t $.  For any $\pi  \in G_{L,k} \left( {\nu ,\beta ,\lambda } \right)$ we have $\pi \left( {b_k ^\lambda  } \right) = b_k ^\lambda  $.  Also there exists $\sigma  \in G_\nu  $ such that $\sigma \beta  = \beta \pi $.  Then for any $i \in \beta ^{ - 1} \left( {b_j ^\nu  } \right)$ we have $\beta \pi \left( i \right) = \sigma \beta \left( i \right) \in b_j ^\nu  $, so $\pi \left( {\beta ^{ - 1} \left( {b_j ^\nu  } \right)} \right) = \beta ^{ - 1} \left( {b_j ^\nu  } \right)$.  Then $\pi \left( {A_i } \right) = A_i \,,\,i = 1,2$.  This means $G_{L,k} \left( {\nu ,\beta ,\lambda } \right)$ lies in a subgroup $\mathfrak{S}_{A_i }  * \mathfrak{S}_{A_2 } $ of $\mathfrak{S}_{\lambda _k } $ of order $a_1 !\,a_2 !$.  Then we have $a_1 !\,a_2 ! = o\left( {G_{L,k} \left( \nu ,\beta ,\lambda \right)} \right) \cdot d$ for some integer $d$.   Also recall that $o\left( {\mathfrak{S}_{\lambda _k } } \right) = \lambda _k ! = p^t !$.  Then compute $n_k  = \frac{{o\left( {\mathfrak{S}_{\lambda _k } } \right)}}{{o\left( {G_{L,k} \left( {\nu ,\beta ,\lambda } \right)} \right)}} = \frac{{p^t !}}
{{a_1 !a_2 !/d}} = d \cdot \frac{{p^t !}}
{{a_1 !\left( {p^t  - a_1 } \right)!}} = d \cdot \left( {\begin{array}{*{20}c}
   {p^t }  \\
   {a_1 }  \\

 \end{array} } \right)$.  Since $0 < a_1  < p^t $, the binomial coefficient $\left( {\begin{array}{*{20}c}
   {p^t }  \\
   {a_1 }  \\

 \end{array} } \right) = 0\bmod p$.  So $n_k  = 0\bmod p$ as desired, and the proof of lemma \ref{l8.2} is complete.
\end{proof}

\begin{proposition}   \label{p8.1}
  The map $\Pi _L  \circ \psi _{\lambda ,L} :k\left[ {S_{r,L} ^\lambda  } \right] \to C_L ^\lambda  $ is an isomorphism of algebras.  Thus as a $k$-algebra, $C_L ^\lambda  $ is isomorphic to a tensor product of monoid algebras, $C_L ^\lambda   \cong k\left[ {S_0 } \right] \otimes \left( {\mathop  \otimes \limits_{i > 0} \left( {k\left[ {\mathfrak{S}_{s_i } } \right]} \right)} \right)$, where $S_0 $ is some submonoid of $\bar \tau _{s_0 } $ containing $\mathfrak{S}_{s_0 } $.  For $S_r  = \mathfrak{S}_r \,,\,\tau _r \,,\,\Re _r \,,\,\bar \tau _r $ we have $S_0  = \mathfrak{S}_{s_0 } \,,\,\tau _{s_0 } \,,\,\Re _{s_0 } \,,\,\bar \tau _{s_0 } $ respectively.
\end{proposition}

\begin{proof}
  $\Pi _L  \circ \psi _{\lambda ,L} :k\left[ {S_{r,L} ^\lambda  } \right] \to C_L ^\lambda  $ is surjective by proposition \ref{p7.4}.  Also $\psi _{\lambda ,L} :k\left[ {S_{r,L} ^\lambda  } \right] \to B_L ^\lambda  $ is an injective algebra homorphism by proposition \ref{p7.2}.  But by lemma \ref{l8.2}, $\ker \left( {\Pi _L } \right) \cap image\left( {\psi _{\lambda ,L} } \right) = \left\{ 0 \right\}$, so $\ker \left( {\Pi _L  \circ \psi _{\lambda ,L} } \right) = \left\{ 0 \right\}$.  Then $\Pi _L  \circ \psi _{\lambda ,L} :k\left[ {S_{r,L} ^\lambda  } \right] \to C_L ^\lambda  $ is an isomorphism of algebras as claimed.

     As a monoid, $S_{r,L} ^\lambda  $ is isomorphic to a product monoid $S_0  \cdot \prod\limits_{i > 0} {\mathfrak{S}_{s_i } } $, with $S_0 $ as described, so $C_L ^\lambda   \cong k\left[ {S_{r,L} ^\lambda  } \right] \cong k\left[ {S_0 } \right] \otimes \left( {\mathop  \otimes \limits_{i > 0} \left( {k\left[ {\mathfrak{S}_{s_i } } \right]} \right)} \right)$.
\end{proof}

As shown in \cite{May1}, the (isomorphism classes of) irreducible representations of the tensor product $k\left[ {S_0 } \right] \otimes \left( {\mathop  \otimes \limits_{i > 0} \left( {k\left[ {\mathfrak{S}_{s_i } } \right]} \right)} \right)$ correspond one to one with a choice of (classes of) irreducible representations of $S_0 $ and each $\mathfrak{S}_{s_i } $.  The irreducible representations of $\mathfrak{S}_{s_i } $ correspond to $p$-regular partitions of $s_i $, while the irreducible representations of $S_0 $ correspond to $p$-regular partitions of $l$ for certain integers $0 \leqslant l \leqslant s_0 $.  We then get the following ``classification theorem'' for the irreducible representations of $B_L $:

\begin{thm}   \label{t8.1}
  Let $k$ be a field of characteristic $p$ and assume $S_r  = \mathfrak{S}_r \,,\,\tau _r $ or any monoid containing the rook monoid $\Re _r $.  There is one isomorphim class of irreducible $B_L $-modules for each choice of the following data:

1.  a decomposition $r = \sum\nolimits_{i \geqslant 0} {s_i p^i } $ for integers $s_i  \geqslant 0$, and

2.  a $p$-regular partition of $s_i $ for each $s_i  > 0\,,\,i > 0$, and

3.  a $p$-regular partition of $s_0 $ when $S_r  = \mathfrak{S}_r $, or

    an integer $j$ with $1 \leqslant j \leqslant s_0 $
 and a $p$-regular partition of $j$ when $S_r  = \tau _r $, or

     an integer $j$ with $0 \leqslant j \leqslant s_0 $
    and a $p$-regular partition of $j$ if $j > 0$ when $S_r $ contains   the rook monoid $\Re _r $.
\end{thm}

     Now turn to the case $B_R $.  Given a partition $\lambda  \in \Lambda ^ +  $, a composition $\nu  \in \Lambda $ and any $\alpha  \in S_r $, consider the subgroup $G_R \left( {\lambda ,\alpha ,\nu } \right) \subseteq \mathfrak{S}_\lambda  $.  Let $G_{R,i} \left( {\lambda ,\alpha ,\nu } \right) = G_R \left( {\lambda ,\alpha ,\nu } \right) \cap \mathfrak{S}_{\lambda _i } $.  Then $\prod\limits_i {G_{R,i} \left( {\lambda ,\alpha ,\nu } \right)} $ is a direct product of disjoint subgroups of $G_R \left( {\lambda ,\alpha ,\nu } \right)$.  Corresponding to lemma \ref{l8.1}, we have

\begin{lemma}    \label{l8.3}
  For any $\alpha  \in S_r $, $G_R \left( {\lambda ,\alpha ,\nu } \right) = \prod\limits_i {G_{R,i} \left( {\lambda ,\alpha ,\nu } \right)} $.
\end{lemma}

\begin{proof}
  Take any $\sigma  \in G_R \left( {\lambda ,\alpha ,\nu } \right) \subseteq \mathfrak{S}_\lambda  $ and write $\sigma  = \prod {\sigma _i } \,,\,\sigma _i  \in \mathfrak{S}_{\lambda _i } $, where $\sigma _i \left( j \right) = \left\{ {\begin{array}{*{20}c}
   {\sigma \left( j \right){\text{ if }}j \in b_i ^\lambda  }  \\
   {j{\text{ otherwise}}}  \\
 \end{array} } \right.$.  We claim each $\sigma _i  \in G_{R,i} \left( {\lambda ,\alpha ,\nu } \right)$, and therefore $\sigma  \in \prod\limits_i {G_{R,i} \left( {\lambda ,\alpha ,\nu } \right)} $
 proving the lemma.  To see that $\sigma _i  \in G_{R,i} \left( {\lambda ,\alpha ,\nu } \right)$, let $A_j  = \alpha ^{ - 1} \left( {b_j ^\lambda  } \right)\,,\,j = 1,2, \cdots ,r$, and $A_0  = \alpha ^{ - 1} \left( 0 \right) - \left\{ 0 \right\}$.  Then $A_j \,,\,0 \leqslant j \leqslant r$, gives a partition of $\left\{ {1,2, \cdots ,r} \right\}$
into disjoint subsets.  Since $\sigma  \in G_R \left( {\lambda ,\alpha ,\nu } \right)$, there exists $\pi  \in \mathfrak{S}_\nu  $ such that $\sigma \alpha  = \alpha \pi $.  Then $x \in A_j  \Rightarrow \alpha \pi \left( x \right) = \sigma \alpha \left( x \right) \in \sigma b_j ^\lambda   = b_j ^\lambda   \Rightarrow \pi \left( x \right) \in A_j $, so $\pi \left( {A_j } \right) = A_j $ for all $j$.  For $j = 1,2, \cdots ,r$, define $\pi _j  \in \mathfrak{S}_r {\text{ by }}\pi _j \left( x \right) = \left\{ {\begin{array}{*{20}c}
   {\pi \left( x \right)\,,\,x \in A_j }  \\
   {x\,\,\,,\,\,x \notin A_j }  \\
 \end{array} } \right.$.  Suppose $x$ is in a particular $\nu $-block $b_l ^\nu  $.  If $x \in A_j $, then $\pi _j \left( x \right) = \pi \left( x \right) \in b_l ^\nu  $ since $\pi  \in G_\nu  $.  On the other hand, if $x \notin A_j $, then $\pi _j \left( x \right) = x \in b_l ^\nu  $.  Since this is true for every $\nu $-block $b_l ^\nu  $, $\pi _j  \in G_\nu  $.   If $x \in A_i $, then $\alpha \left( x \right) \in b_i ^\lambda  $, so $\sigma _i \alpha (x) = \sigma \alpha \left( x \right) = \alpha \pi \left( x \right) = \alpha \pi _i \left( x \right)$.  On the other hand, if $x \notin A_i $, then $\alpha \left( x \right) \notin b_i ^\lambda  $ and $\sigma _i \alpha (x) = \alpha \left( x \right) = \alpha \pi _i \left( x \right)$.  So we have $\sigma _i \alpha  = \alpha \pi _i $ and $\sigma _i  \in G_{R,i} \left( {\lambda ,\alpha ,\nu } \right)$ as desired. 
\end{proof}

\begin{lemma}    \label{l8.4}
  Suppose $\gamma  \in S_r ^\lambda  $ has the following property:
For any factorization $\gamma  = \alpha \beta \,,\,\alpha ,\beta  \in S_r $ and any composition $\nu  < \lambda $, there exists an integer $i$ such that the size $\nu _j $ of the $\nu $-block $b_j ^\nu  $ containing $i$ is less than the size $\lambda _k $ of the $\lambda $-block $b_k ^\lambda  $ containing $\alpha \left( i \right)$.  Then if an element $x \in \ker \left( {\Pi _R } \right) \subseteq B_R ^\lambda  $ is expanded in terms of the double coset basis $\{ f\left( {\lambda ,D,\lambda } \right):D \in \,_\lambda  M_\lambda  \} $ for $B_R ^\lambda  $, $x = \sum\limits_{D_\alpha   \in \,_\lambda  M_\lambda  } {c_\alpha  f\left( {\lambda ,\alpha ,\lambda } \right)} $, then the coefficient $c_\gamma  $ for any $f\left( {\lambda ,\gamma ,\lambda } \right)$ with $\gamma  \in S_{r,L} ^\lambda  $ will be 0.
\end{lemma}

\begin{proof}
  Since $x \in \ker \left( {\Pi _R } \right) \subseteq B_R ^\lambda  e_\lambda  B_R ^\lambda  $, $x$ will be a $k$ linear combination of terms of the form  $f\left( {\lambda ,\alpha ,\nu } \right) * _R f\left( {\nu ,\beta ,\lambda } \right)$ for some $\nu $ with $\nu  < \lambda $.  Then by the alternative form of the multiplication rule in section 3, $f\left( {\lambda ,\alpha ,\nu } \right) * _R f\left( {\nu ,\beta ,\lambda } \right) = \sum\limits_{D \in \,_\lambda  M_\lambda  } {\frac{{n_R (D)\,\,a_{R,D} }}{{n_R \left( {\lambda ,\alpha ,\nu } \right)}}f\left( {\lambda ,D,\lambda } \right)} $ for certain integers $a_{R,D}  \in k$.  Then the coefficient of any basis element $f\left( {\lambda ,\gamma ,\lambda } \right)$ in the expansion of $x$ will be a $k$ linear combination of terms $\frac{{n_R \left( {\lambda ,\gamma ,\lambda } \right)}}{{n_R \left( {\lambda ,\alpha ,\nu } \right)}}$.  We will show that $\frac{{n_R \left( {\lambda ,\gamma ,\lambda } \right)}}
{{n_R \left( {\lambda ,\alpha ,\nu } \right)}} \equiv 0\,\,\bmod (p)$ whenever $\gamma  \in S_r ^\lambda  $ satisfies the hypothesis in the lemma and $a_{R,D_\gamma  }  \ne 0$.  Then $c_\gamma   = 0$ in $k$ for such $\gamma $ and the lemma is proved.

Notice that for any $\gamma  \in S_r ^\lambda  $, $G_R \left( {\lambda ,\gamma ,\lambda } \right) = G_\lambda   = \prod\limits_i {G_{\lambda _i } } $.  By lemma \ref{l8.3}, $G_R \left( {\lambda ,\alpha ,\nu } \right) = \prod\limits_i {G_{R,i} \left( {\lambda ,\alpha ,\nu } \right)} $ , where each $G_{R,i} \left( {\lambda ,\alpha ,\nu } \right)$ is a subgroup of $G_{\lambda _i } $.  Write $n_i \left( {\lambda ,\alpha ,\nu } \right) = o\left( {G_{R,i} \left( {\lambda ,\alpha ,\nu } \right)} \right)$, which must be a factor of $o\left( {G_{\lambda _i } } \right) = \lambda _i !$.  Then $\frac{{n_R \left( {\lambda ,\gamma ,\lambda } \right)}}
{{n_R \left( {\lambda ,\alpha ,\nu } \right)}} = \prod\limits_i {\frac{{o\left( {G_{\lambda _i } } \right)}}{{n_i \left( {\lambda ,\alpha ,\nu } \right)}}} $.  If we show that $n_i  \equiv \frac{{o\left( {G_{\lambda _i } } \right)}}{{n_i \left( {\lambda ,\alpha ,\nu } \right)}} \equiv 0\,\,\bmod (p)$ for at least one $i$, then $c_\gamma   = 0$ and we are done.

     If  $a_{R,D_\gamma  }  \ne 0$ we can assume $\gamma  = \alpha \rho \beta $ for some $\rho  \in \mathfrak{S}_\nu  $.  By our hypothesis, there will be an integer $i$ in a block $b_j ^\nu  $ of size $\nu _j  = p^s $ such that $\alpha \left( i \right) \in b_k ^\lambda  $ for some block $b_k ^\lambda  $ of size $\lambda _k  = p^t  > p^s $.  Let $A_{\,1}  = b_k ^\lambda   \cap \alpha \left( {b_j ^\nu  } \right)\,,\,A_{\,2}  = b_k ^\lambda   - A_{\,1} \,,\,a_i  = \# A_i $.  Then $1 \leqslant a_1  \leqslant p^s  < p^t $ and $a_1  + a_2  = \# b_k ^\lambda   = \lambda _k  = p^t $.  For any $\sigma  \in G_{R,k} \left( {\lambda ,\alpha ,\nu } \right) \subseteq G_\lambda  $ we have $\sigma \left( {b_k ^\lambda  } \right) = b_k ^\lambda  $.  Also, there exists $\pi  \in G_\nu  $ such that $\sigma \alpha  = \alpha \pi $, so $\sigma \left( {\alpha \left( {b_j ^\nu  } \right)} \right) = \alpha \pi \left( {b_j ^\nu  } \right) = \alpha \left( {b_j ^\nu  } \right)$.  Then $\sigma \left( {A_i } \right) = A_i \,,\,i = 1,2$.  This means $G_{R,k} \left( {\lambda ,\alpha ,\nu } \right)$ lies in a subgroup $\mathfrak{S}_{A_i }  * \mathfrak{S}_{A_2 } $ of $\mathfrak{S}_{\lambda _k } $ of order $a_1 !\,a_2 !$.  Then we have $a_1 !\,a_2 ! = o\left( {G_{R,k} \left( {\lambda ,\alpha ,\nu } \right)} \right) \cdot d$ for some integer $d$.   Also recall that $o\left( {\mathfrak{S}_{\lambda _k } } \right) = \lambda _k ! = p^t !$.  Then compute $n_k  = \frac{{o\left( {\mathfrak{S}_{\lambda _k } } \right)}}
{{o\left( {G_{R,k} \left( {\lambda ,\alpha ,\nu } \right)} \right)}} = \frac{{p^t !}}
{{a_1 !a_2 !/d}} = d \cdot \frac{{p^t !}}
{{a_1 !\left( {p^t  - a_1 } \right)!}} = d \cdot \left( {\begin{array}{*{20}c}
   {p^t }  \\
   {a_1 }  \\
 \end{array} } \right)$.  Since $0 < a_1  < p^t $, the binomial coefficient $\left( {\begin{array}{*{20}c}
   {p^t }  \\
   {a_1 }  \\
 \end{array} } \right) = 0\bmod p$.  So $n_k  = 0\bmod p$ as desired, and the proof of lemma \ref{l8.4} is complete.
\end{proof}

     We now consider several special cases of the monoid $S_r $. 
\\

\textbf{The case $S_r  = \mathfrak{S}_r $}
\\
 
 The simplest case is when $S_r  = \mathfrak{S}_r $ and $B_R  = S_k \left( {r,n} \right)$, the standard Schur algebra over $k$.   In this case every element of $S_r ^\lambda   \subseteq \mathfrak{S}_r $ must be one to one.  So $\phi _\lambda  ^{ - 1} \left( {S_r ^\lambda  } \right) = \prod\limits_{s_i  > 0} {\mathfrak{S}_{s_i } } $ and then (since $\phi _\lambda  $ is injective) $S_r ^\lambda   = \phi _\lambda  \left( {\prod\limits_{s_i  > 0} {\mathfrak{S}_{s_i } } } \right)$.

\begin{proposition}    \label{p8.2}
  For $S_r  = \mathfrak{S}_r $, the map $\Pi _R  \circ \psi _{\lambda ,R} :k\left[ {S_r ^\lambda  } \right] \to C_R ^\lambda  $ is an isomorphism of $k$-algebras.  There is a $k$-algebra isomorphism  $C_R ^\lambda   \cong \mathop  \otimes \limits_{s_i  > 0} \left( {k\left[ {\mathfrak{S}_{s_i } } \right]} \right).$
\end{proposition}

\begin{proof}
  By proposition \ref{p7.2}, $\Pi _R  \circ \psi _{\lambda ,R} :k\left[ {S_r ^\lambda  } \right] \to C_R ^\lambda  $ is surjective.  We will show that any element $\gamma  \in S_r ^\lambda  $ satisfies the hypothesis of lemma \ref{l8.4}, so $\ker \left( {\Pi _R } \right) \cap image\left( {\psi _{\lambda ,R} } \right) = \left\{ 0 \right\}$.  Since $\psi _{\lambda ,R} :k\left[ {S_r ^\lambda  } \right] \to B_R ^\lambda  $ is injective algebra homorphism by proposition \ref{p7.1}, $\ker \left( {\Pi _R  \circ \psi _{\lambda ,R} } \right) = \left\{ 0 \right\}$.  Then $\Pi _R  \circ \psi _{\lambda ,R} :k\left[ {S_{r,R} ^\lambda  } \right] \to C_R ^\lambda  $ is an isomorphism of algebras as claimed.

     Given any $\gamma  \in S_r ^\lambda  $, consider a factorization $\gamma  = \alpha \beta \,,\,\alpha ,\beta  \in S_r  = \mathfrak{S}_r $ and any composition $\nu  < \lambda $.  Notice that $\alpha  \in \mathfrak{S}_r  \Rightarrow \alpha $ is surjective.  Since $\nu  < \lambda $, there must be some length $l$ such that $\lambda ,\nu $ have the same number of rows of any length $ > l$, while $\lambda $ has more rows of length $l$ than $\nu $ does.  Thus there are more integers in $\lambda $-blocks of size $ \geqslant l$ than in $\nu $-blocks of size $ \geqslant l$.  Then since $\alpha $ is surjective, some element in a $\lambda $-block of size $ \geqslant l$ must be the image under $\alpha $ of an element in a smaller $\nu $-block of size $ < l$.  So $\gamma $ does satisfy the hypothesis of lemma \ref{l8.4}.
\end{proof}       

     As shown in \cite{May1}, the (isomorphism classes of) irreducible representations of the tensor product $C_R ^\lambda   \cong \mathop  \otimes \limits_{s_i  > 0} \left( {k\left[ {\mathfrak{S}_{s_i } } \right]} \right)$ correspond one to one with a choice of (classes of) irreducible representations each $\mathfrak{S}_{s_i } $.  This leads to the following
\\
  
\textbf{Classification Theorem for $B_R  \cong S_k (r,n)$:}

  Let $k$ be a field of positive characteristic $p$ and let $S_r  = \mathfrak{S}_r $.  There is one isomorphism class of irreducible $B_R $
-modules for each choice of the following data:

1.  a decomposition $r = \sum\limits_{i \geqslant 0} {s_i p^i } $ for integers $s_i  \geqslant 0$ and

2.  a $p$-regular partition of $s_i $ for each $s_i  > 0$.
\\

Evidently $s_i  = 0$ for all but a finite number of $i$.  The usual classification theorem for the Schur algebra matches irreducible modules with arbitrary partitions of $r$.  It is a pleasant combinatorial exercise to match arbitrary partitions of $r$ with choices of data as in 1. and 2. above.
\\

\textbf{The case $S_r  = \tau _r $.}  
\\

     Let $\lambda $ be a $p$-partition and let $m$ and $M$ be the smallest and largest integers such that $s_i  \ne 0$.   Then $r = \sum\limits_{i = m}^M {s_i p^i } $ and $S_r ^\lambda   = \phi _\lambda  \left( {\prod\limits_{i = m}^M {\tau _{s_i } } } \right)$.  Let $S' \subseteq S_r ^\lambda  $ be the subsemigroup $S' \equiv \phi _\lambda  \left( {\tau _{s_m }  \cdot \prod\limits_{i = m + 1}^M {\mathfrak{S}_{s_i } } } \right)$. 

\begin{lemma}    \label{l8.5}
  $\Pi _R  \circ \psi _{\lambda ,R} :k\left[ {S'} \right] \to C_R ^\lambda  $ is surjective.
\end{lemma}

\begin{proof}
  By proposition \ref{p7.2}, $\Pi _R  \circ \psi _{\lambda ,R} :k\left[ {S_r ^\lambda  } \right] \to C_R ^\lambda  $ is surjective.  Take $\gamma  \in S_r ^\lambda   - S'$.  We will show that $f\left( {\lambda ,\gamma ,\lambda } \right)  = \psi _{\lambda ,R} \left( \gamma  \right) \in \ker \left( {\Pi _R } \right)$.  Then $\Pi _R  \circ \psi _{\lambda ,R} :k\left[ {S'} \right] \to C_R ^\lambda  $ is surjective as claimed.  If $\gamma  \in S_r ^\lambda   - S'$, then there must be some $k > m$ and a $\lambda $-block $b_i ^\lambda  $ of size $\lambda _i  = p^k  > p^m $ such that $image\left( \gamma  \right) \cap b_i ^\lambda   = \emptyset $.  Let $\nu $ be the composition obtained from $\lambda $ by replacing the $\lambda $-block $b_i ^\lambda  $ by $p^{k - m} $ $\nu $-blocks of size $p^m $.  Then $\nu  < \lambda $ and $\mathfrak{S}_\nu   \subseteq \mathfrak{S}_\lambda  $.  Define $\beta  \in \tau _r $ by letting $\beta \left( j \right) = j$ for all integers $j$ outside of the $\lambda $-block $b_i ^\lambda  $, while $\beta $ maps each of the new $\nu $-blocks one to one onto a $\lambda $-block of size $p^m $.  From the construction, we have $\beta \gamma  = \gamma $ and it is easy to check that $G_R \left( {\lambda ,\beta ,\nu } \right) = G_\lambda  $ so $n_R (\lambda ,\beta ,\nu ) = o\left( {\mathfrak{S}_\lambda  } \right)$.  Since $\gamma  \in S_r ^\lambda  $, we have $G_R \left( {\lambda ,\gamma ,\lambda } \right) = G_\lambda  $ and since $\mathfrak{S}_\nu   \subseteq \mathfrak{S}_\lambda  $, we have $G_R \left( {\nu ,\gamma ,\lambda } \right) = G_R \left( {\lambda ,\gamma ,\lambda } \right) \cap G_\nu   = G_\nu  $.  So $n_R (\lambda ,\gamma ,\lambda ) = o\left( {\mathfrak{S}_\lambda  } \right)$ and  $n_R (\nu ,\gamma ,\lambda ) = o\left( {\mathfrak{S}_\nu  } \right)$.  Finally, for any $\rho  \in \mathfrak{S}_\nu  $ there is a $\pi  \in G_\lambda  $ such that $\rho \gamma  = \gamma \pi $.  Then $\beta \rho \gamma  = \beta \gamma \pi  = \gamma \pi  \in D_\gamma  $.  So $N\left( {D_\beta  ,D_\gamma  ,D} \right) = o\left( {G_\nu  } \right)$ if $D = D_\gamma  $ and is $0$ otherwise.  The multiplication rule then gives:
\[f\left( {\lambda ,\beta ,\nu } \right) * _R f\left( {\nu ,\gamma ,\lambda } \right) = \frac{{o\left( {G_\nu  } \right) \cdot n_R (\lambda ,\gamma ,\lambda )}}
{{n_R (\lambda ,\beta ,\nu )n_R (\nu ,\gamma ,\lambda )}} \cdot f\left( {\lambda ,\gamma ,\lambda } \right)\]
\[ = \frac{{o\left( {G_\nu  } \right) \cdot o\left( {G_\lambda  } \right)}}
{{o\left( {G_\nu  } \right) \cdot o\left( {G_\lambda  } \right)}} \cdot f\left( {\lambda ,\gamma ,\lambda } \right) = f\left( {\lambda ,\gamma ,\lambda } \right).\]  So $f\left( {\lambda ,\gamma ,\lambda } \right) = \psi _{\lambda ,R} \left( \gamma  \right) \in \ker \left( {\Pi _R } \right)$ as claimed.  
\end{proof}

\begin{lemma}     \label{l8.6}
  $\Pi _R  \circ \psi _{\lambda ,R} :k\left[ {S'} \right] \to C_R ^\lambda  $ is injective.
\end{lemma}

\begin{proof}
  We will show that any element $\gamma  \in S'$ satisfies the hypothesis of lemma \ref{l8.4}, so $\ker \left( {\Pi _R } \right) \cap image\left( {\psi _{\lambda ,R} \left| {k\left[ {S'} \right]} \right.} \right) = \left\{ 0 \right\}$.  Since $\psi _{\lambda ,R} :k\left[ {S'} \right] \to B_R ^\lambda  $ is an injective algebra homorphism by proposition \ref{p7.1}, $\ker \left( {\Pi _R  \circ \psi _{\lambda ,R} \left| {k\left[ {S'} \right]} \right.} \right) = \left\{ 0 \right\}$ and $\Pi _R  \circ \psi _{\lambda ,R} :k\left[ {S'} \right] \to C_R ^\lambda  $ is injective.   

     Take any $\gamma  \in S' \subseteq S_r ^\lambda  $, any factorization $\gamma  = \alpha \beta \,,\,\alpha ,\beta  \in \tau _r $, and any composition $\nu  < \lambda $.  If there is any $\nu $-block of size less than the smallest $\lambda $-block size $p^m $, then for any $i$ in such a $\nu $-block, $\alpha \left( i \right)$ must lie in a larger size $\lambda $-block, so the hypothesis of lemma \ref{l8.4} is satisfied.  If all $\nu $-blocks have size $ \geqslant p^m $, then $\nu  < \lambda $ implies there exists an integer $k > p^m $ such that $L\left( {\nu ,i} \right) = L\left( {\lambda ,i} \right){\text{ for }}i > k$, while $L\left( {\nu ,k} \right) < L\left( {\lambda ,k} \right)$.  So there are more integers in $\lambda $-blocks of size $ \geqslant k > p^m $ than in $\nu $-blocks of size $ \geqslant k$.  However, for $\gamma  \in S'$, any integer in a $\lambda $-block of size $ > p^m $ is in $image\left( \gamma  \right)$ and therefore also in $image\left( \alpha  \right)$.  It follows that there must be some integer $i$ in a $\nu $-block of size $ < k$ such that $\alpha \left( i \right)$ is in a $\lambda $-block of size $ \geqslant k$.  So the hypothesis of lemma \ref{l8.4} is again satisfied and the proof is complete. 
\end{proof}

     Combining lemmas \ref{l8.5} and \ref{l8.6} gives

\begin{proposition}    \label{p8.3}
  $\Pi _R  \circ \psi _{\lambda ,R} :k\left[ {S'} \right] \to C_R ^\lambda  $ is a isomorphism of $k$-algebras.  There are isomorphisms of $k$-algebras $C_R ^\lambda   \cong k[S'] \cong k\left[ {\tau _{s_m }  \cdot \prod\limits_{i = m + 1}^M {\mathfrak{S}_{s_i } } } \right] \cong k\left[ {\tau _{s_m } } \right] \otimes \left( {\mathop  \otimes \limits_{i = m + 1}^M k\left[ {\mathfrak{S}_{s_i } } \right]} \right)$ .
\end{proposition}

     So irreducible $B_R $-modules at level $\lambda $ correspond to irreducible $C_R ^\lambda  $-modules, which correspond to a choice of an irreducible $k\left[ {\tau _{s_m } } \right]$ module and irreducible $k\left[ {\mathfrak{S}_{s_i } } \right]$ modules for each $m + 1 \leqslant i \leqslant M$.  These in turn are classified by a $p$-regular partition of $j$ for some $1 \leqslant j \leqslant s_m $ and $p$-regular partitions of $s_i $ for each $s_i  > 0\,,\,i > m$.  We can now state the 
\\

\textbf{Classification Theorem for $B_R \,,\,S_r  = \tau _r $:}

Let $k$ be a field of positive characteristic $p$ and $S_r  = \tau _r $.  There is one isomorphism class of irreducible $B_R $-modules for each choice of the following data:  

1.  a decomposition $r = \sum\limits_{i \geqslant m} {s_i p^i } $ for integers $s_i  \geqslant 0\,,\,s_m  > 0$ and

2.  a $p$-regular partition of $s_i $ for each $s_i  > 0\,,\,i > m$ and

3.  an integer $j$ with $1 \leqslant j \leqslant s_m $ and

4.  a $p$-regular partition of $j$.
\\

An irreducible $B$-module corresponding to such data will be at level $\lambda $ where $\lambda $ is the partition with $s_i $ blocks of size $p^i $ and will have index $\bar i = \sum\limits_{i > m} {s_i p^i }  + jp^m $.

Notice that for $p < r$ we must have $s_i  = 0\,,\,\forall i > 0$, so we have $C_R ^\lambda   = \left\{ {\begin{array}{*{20}c}
   {k\left[ {\tau _r } \right]\,,\,\lambda  = \bar \nu }  \\
   {0\,,\,\,\,\,\,\,\lambda  > \bar \nu }  \\
 \end{array} } \right.$.  So irreducible $B$-modules correspond to irreducible $k\left[ {\tau _r } \right]$ modules as shown in \cite{MA}.
\\

\textbf{The case $S_r  \supseteq \Re _r $}   
\\

In this section we will assume that $S_r  \supseteq \Re _r $ (the rook algebra) and that $k$ is a field of positive characteristic $p$.  For example, we could have $S_r  = \Re _r $ or $S_r  = \bar \tau _r $.  Our analysis will follow the pattern for the case $S_r  = \tau _r $ given above.  Let $\lambda $ be a $p$-partition.  Let $S' \subseteq S_r ^\lambda  $ be the subsemigroup $S' \equiv \phi _\lambda  \left( {\bar \tau _{s_{\,0} }  \cdot \prod\limits_{i > 0}^{} {\mathfrak{S}_{s_i } } } \right) \cap S_r ^\lambda   = \phi _\lambda  \left( {S_0  \cdot \prod\limits_{i > 0}^{} {\mathfrak{S}_{s_i } } } \right)$ where $S_0 $ is a semigroup with $\bar \tau _{s_0 }\supseteq S_0 \supseteq \Re _{s_0 }$.  (When $S_r  = \Re _r $ we have $S_0  = \Re _{s_0 } $; when $S_r  = \bar \tau _r $ we have $S_0  = \bar \tau _{s_0 } $ .) 

\begin{lemma}    \label{l8.7}
  $\Pi _R  \circ \psi _{\lambda ,R} :k\left[ {S'} \right] \to C_R ^\lambda  $ is surjective.
\end{lemma}

\begin{proof}
  By proposition \ref{p7.2}, $\Pi _R  \circ \psi _{\lambda ,R} :k\left[ {S_r ^\lambda  } \right] \to C_R ^\lambda  $ is surjective.  Take $\gamma  \in S_r ^\lambda   - S'$.  We will show that $f\left( {\lambda ,\gamma ,\lambda } \right) = \psi _{\lambda ,R} \left( \gamma  \right) \in \ker \left( {\Pi _R } \right)$.  Then $\Pi _R  \circ \psi _{\lambda ,R} :k\left[ {S'} \right] \to C_R ^\lambda  $ is surjective as claimed.  If $\gamma  \in S_r ^\lambda   - S'$, then there must be some $k > 0$ and a $\lambda $-block $b_i ^\lambda  $ of size $\lambda _i  = p^k  > 1$ such that $image\left( \gamma  \right) \cap b_i ^\lambda   = \emptyset $.  Let $\nu $ be the composition obtained from $\lambda $ by replacing the $\lambda $-block $b_i ^\lambda  $ by $p^k $ $\nu $-blocks of size $1$.  Then $\nu  < \lambda $ and $G_\nu   \subseteq G_\lambda  $.  Define $\beta  \in \Re _r  \subseteq S_r $ by letting $\beta\left(j\right) = j$ for all integers $j$ outside of the $\lambda $-block $b_i ^\lambda  $, while $\beta $ maps each of the new $\nu $-blocks to 0.  From the construction, we have $\beta \gamma  = \gamma $ and it is easy to check that $G_R \left( {\lambda ,\beta ,\nu } \right) = G_\lambda  $ so $n_R (\lambda ,\beta ,\nu ) = o\left( {\mathfrak{S}_\lambda  } \right)$.  Since $\gamma  \in S_r ^\lambda  $, we have $G_R \left( {\lambda ,\gamma ,\lambda } \right) = G_\lambda  $ and since $G_\nu   \subseteq G_\lambda  $, we have $G_R \left( {\nu ,\gamma ,\lambda } \right) = G_R \left( {\lambda ,\gamma ,\lambda } \right) \cap G_\nu   = G_\nu  $.  So $n_R (\lambda ,\gamma ,\lambda ) = o\left( {\mathfrak{S}_\lambda  } \right)$ and  $n_R (\nu ,\gamma ,\lambda ) = o\left( {\mathfrak{S}_\nu  } \right)$.  Finally, for any $\rho  \in \mathfrak{S}_\nu  $ there is a $\pi  \in G_\lambda  $ such that $\rho \gamma  = \gamma \pi $.  Then $\beta \rho \gamma  = \beta \gamma \pi  = \gamma \pi  \in D_\gamma  $.  So $N\left( {D_\beta  ,D_\gamma  ,D} \right) = o\left( {G_\nu  } \right)$ if $D = D_\gamma  $ and is $0$ otherwise.  The multiplication rule then gives:
\[f\left( {\lambda ,\beta ,\nu } \right) * _R f\left( {\nu ,\gamma ,\lambda } \right) = \frac{{o\left( {G_\nu  } \right) \cdot n_R (\lambda ,\gamma ,\lambda )}}
{{n_R (\lambda ,\beta ,\nu )n_R (\nu ,\gamma ,\lambda )}} \cdot f\left( {\lambda ,\gamma ,\lambda } \right)\]
\[= \frac{{o\left( {G_\nu  } \right) \cdot o\left( {G_\lambda  } \right)}}
{{o\left( {G_\nu  } \right) \cdot o\left( {G_\lambda  } \right)}} \cdot f\left( {\lambda ,\gamma ,\lambda } \right) = f\left( {\lambda ,\gamma ,\lambda } \right).\]
So $f\left( {\lambda ,\gamma ,\lambda } \right) = \psi _{\lambda ,R} \left( \gamma  \right) \in \ker \left( {\Pi _R } \right)$ as claimed.  
\end{proof}

\begin{lemma}    \label{l8.8}
  $\Pi _R  \circ \psi _{\lambda ,R} :k\left[ {S'} \right] \to C_R ^\lambda  $ is injective.
\end{lemma}

\begin{proof}
  We will show that any element $\gamma  \in S'$ satisfies the hypothesis of lemma \ref{l8.4}, so $\ker \left( {\Pi _R } \right) \cap image\left( {\psi _{\lambda ,R} \left| {k\left[ {S'} \right]} \right.} \right) = \left\{ 0 \right\}$.  Since $\psi _{\lambda ,R} :k\left[ {S'} \right] \to B_R ^\lambda  $ is an injective algebra homorphism by proposition \ref{p7.1}, $\ker \left( {\Pi _R  \circ \psi _{\lambda ,R} \left| {k\left[ {S'} \right]} \right.} \right) = \left\{ 0 \right\}$ and $\Pi _R  \circ \psi _{\lambda ,R} :k\left[ {S'} \right] \to C_R ^\lambda  $ is injective.   

     Take any $\gamma  \in S' \subseteq S_r ^\lambda  $, any factorization $\gamma  = \alpha \beta \,,\,\alpha ,\beta  \in \tau _r $, and any composition $\nu  < \lambda $.  Since $\nu  < \lambda $,  there exists an integer $k > 1$ such that $L\left( {\nu ,i} \right) = L\left( {\lambda ,i} \right){\text{ for }}i > k$, while $L\left( {\nu ,k} \right) < L\left( {\lambda ,k} \right)$.  So there are more integers in $\lambda $-blocks of size $ \geqslant k > 1$ than in $\nu $-blocks of size $ \geqslant k$.  However, for $\gamma  \in S'$, any integer in a $\lambda $-block of size $ > 1$ is in $image\left( \gamma  \right)$ and therefore also in $image\left( \alpha  \right)$.  It follows that there must be some integer $i$ in a $\nu $-block of size $ < k$ such that $\alpha \left( i \right)$ is in a $\lambda $-block of size $ \geqslant k$.  So the hypothesis of lemma \ref{l8.4} is again satisfied and the proof is complete.  
\end{proof}

     Combining lemmas \ref{l8.7} and \ref{l8.8} gives

\begin{proposition}     \label{p8.4}
  $\Pi _R  \circ \psi _{\lambda ,R} :k\left[ {S'} \right] \to C_R ^\lambda  $ is a isomorphism of $k$-algebras.  There are isomorphisms of $k$-algebras $C_R ^\lambda   \cong k[S'] \cong k\left[ {S_0  \cdot \prod\limits_{i > 0}^{} {\mathfrak{S}_{s_i } } } \right] \cong k\left[ {S_0 } \right] \otimes \left( {\mathop  \otimes \limits_{i > 0}^{} k\left[ {\mathfrak{S}_{s_i } } \right]} \right)$ .
\end{proposition}

    Since $\bar \tau _{s_0 }  \supseteq S_0  \supseteq \Re _{s_0 } $, the irreducible representations of $k\left[ {S_0 } \right]$ correspond to irreducible $k\left[ {\mathfrak{S}_j } \right]$ modules for some index $1 \leqslant j \leqslant s_0 $ or to the trivial one dimensional representation of $k\left[ {\Re _{s_0 } } \right]$ of index 0.  So as for the $S_r  = \tau _r $ case we have the 
\\

\textbf{Classification Theorem for $B_R $ when $S_r  \supseteq \Re _r $:}

Assume $S_r  \supseteq \Re _r $ and let $k$ be a field of positive characteristic $p$.  There is one isomorphism class of irreducible $B_R $
-modules for each choice of the following data:  

1.  a decomposition $r = \sum\limits_{i \geqslant 0} {s_i p^i } $ for integers $s_i  \geqslant 0\,$ and

2.  a $p$-regular partition of $s_i $ for each $s_i  > 0\,,\,i > 0$ and

3.  an integer $j$ with $0 \leqslant j \leqslant s_0 $ and

4.  a $p$-regular partition of $j$ if $j > 0$. 
\\

An irreducible $B$-module corresponding to such data will be at level $\lambda $ where $\lambda $ is the partition with $s_i $ blocks of size $p^i $ and will have index $\bar i = \sum\limits_{i > 0} {s_i p^i }  + j$.

\end{document}